\documentclass[10pt]{article}
\usepackage{fullpage,epsfig,graphics,cancel,bbm,slashed}
\usepackage{amsthm,amssymb,amsmath,amsfonts,amscd,amsbsy}
\usepackage{psfrag,hyperref,color}
\usepackage{graphicx}
\usepackage{wasysym}
\usepackage{mathdots}
\usepackage{mathrsfs}

\usepackage{tikz}
\usetikzlibrary{arrows,chains,matrix,positioning,scopes,snakes}

\makeatletter
\tikzset{join/.code=\tikzset{after node path={%
\ifx\tikzchainprevious\pgfutil@empty\else(\tikzchainprevious)%
edge[every join]#1(\tikzchaincurrent)\fi}}}
\makeatother

\tikzset{>=stealth',every on chain/.append style={join},
         every join/.style={->}}
\tikzset{
    >=stealth',
    punkt/.style={
           rectangle,
           rounded corners,
           draw=black, very thick,
           text width=6.5em,
           minimum height=2em,
           text centered},
    pil/.style={
           ->,
           thick,
           shorten <=2pt,
           shorten >=2pt,}
}

\newcommand{\BB}{\mathbb}

\def\tf{t_{\infty}}

\newcommand{\bea}{\begin{eqnarray}}
\newcommand{\eea}{\end{eqnarray}}
\newcommand{\nn}{\nonumber}
\newcommand{\opn}{\operatorname}

\newcommand{\bra}{\langle}
\newcommand{\ket}{\rangle}

\newcommand{\im}{\opn{Im}}
\newcommand{\re}{\opn{Re}\,}
\newcommand{\To}{\Rightarrow}

\def\Gd{\Delta}
\def\ep{\epsilon}
\def\gt{\theta}
\def\vgt{\vartheta}

\def\gs{\sigma}
\def\Gs{\Sigma}
\def\vgs{\varsigma}
\def\ep{\epsilon}

\def\gl{\lambda}

\def\Go{\Omega}
\def\go{\omega}

\DeclareMathAlphabet{\mathpzc}{OT1}{pzc}{m}{it}

\newtheorem{theorem}{Theorem}[section]
\newtheorem{lemma}[theorem]{Lemma}
\newtheorem{proposition}[theorem]{Proposition}

\theoremstyle{definition}
\newtheorem{example}[theorem]{Example}
\newtheorem{remark}[theorem]{Remark}
\newtheorem{corollary}[theorem]{Corollary}
\newtheorem{definition}[theorem]{Definition}

\numberwithin{equation}{section}
\begin{document}
\begin{flushright} \small
UUITP-27/19
\end{flushright}
\smallskip

\begin{center}
\Large Generalised K\"ahler Structure on $\BB{C}P^2$ and Elliptic Functions
\end{center}

\begin{center}
F. Bonechi\footnote{\small INFN Sezione di Firenze, email: francesco.bonechi at fi.infn.it},
J. Qiu\footnote{\small Matematiska institutionen, Institutionen f\"or fysik och astronomi, email: jian.qiu at math.uu.se},
and
M. Tarlini\footnote{\small INFN Sezione di Firenze,  email: marco.tarlini at fi.infn.it}
\end{center}



\begin{abstract}
We construct a toric generalised K\"ahler structure on $\BB{C}P^2$ and show that the various structures such as the complex structure, metric etc are expressed in terms of certain elliptic functions. We also compute the generalised K\"ahler potential in terms of certain integral of elliptic functions.
\end{abstract}

\tableofcontents

\section{Introduction}
We give a biased introduction to generalised K\"ahler structure based on physics applications.
Generalised K\"ahler or Calabi-Yau geometry \cite{GCY} \cite{Gualtieri2014} arise from the study of supersymmetric sigma models or super-string compactifications.
For example, by placing a super-string theory on $\BB{R}^{1,3}\times X$ where $X$ is a Calabi-Yau 3-fold, one can obtain a gauge theory on $\BB{R}^{1,3}$ that preserves $N=2$ supersymmetry. For a sigma model based on mappings from a Riemann surface $\Gs$ to a manifold $M$, Zumino showed that one can obtain $N=2$ supersymmetry if $M$ is K\"ahler \cite{ZUMINO1979203}. But string theory has an important duality called the $T$-duality (torus duality) that relates the IIA to the IIB super-string and is the local model for mirror symmetry \cite{Strominger:1996it}. However the two integral ingredients, the K\"ahler geometry and T-duality are not compatible. The problem is that K\"ahler geometry is characterised by the covariant constancy of the complex structure under the Levi-Civita connection. But the T-duality will generate a Neveu–Schwarz 3-form $H$. This 3-form modifies the Levi-Civita connection, making it torsionful \cite{STROMINGER1986253,HULL1986357}.
So a more appropriate notion of generalised K\"ahler structure is formulated as the constancy of complex structure under a torisonful connection $\nabla^H$ (which we will recall in the next section). As for the sigma models, it turns out that Zumino's result can be generalised \cite{GHR} to accommodate a bi-hermitian geometry with torsion connection.

The local geometry of generalised K\"ahler structure (GKS) has been well understood and used to construct supersymmetric sigma models mentioned above \cite{Lindstrom:2005zr}. The generalised K\"ahler potential (GKP) was used to write the action of such models, but the drawback was that the formulation only works at regular points and a more global understanding of the GKP was lacking.
A more recent work \cite{Bischoff:2018kzk} partially answered this question for GKS of the symplectic type, where the GKP was formulated in terms of Mortia equivalence of certain holomorphic Poisson structures.

In \cite{Hitchin_biherm} Hitchin gave a construction of GKS on Del Pezzo surfaces using a flow that deforms the original K\"ahler structure on these surfaces.
In general, it can be hard to solve a flow equation due to its non-linear nature. In this work we restrict ourselves to the toric Del Pezzo surfaces, where thanks to the toric symmetry, the flow equations greatly simplify. We show that on $\BB{C}P^2$ the flow is solved by the Weierstrass elliptic functions. To present the solution we use the coordinate system $c_3,s$ to parametrise the triangle which is the moment map polygon for $\BB{C}P^2$, see fig.\ref{fig_contour_intro}. If $y^{1,2}$ are the Euclidean coordinates of $\BB{R}^2$, the triangle is the region $1\geq y^{1,2}\geq 0$, $y^1+y^2\leq 1$, we have
  \bea y^k(c_3,s)=\frac{c_3}{\frac{1}{12}-\wp(\go_2+\frac{\go_1}{\pi}(s+\frac{2k\pi}{3}))},\nn\eea
where $\wp$ is the Weierstrass elliptic function satisfying
\bea&\displaystyle \dot{\wp}^2=4\wp^3-g_2\wp-g_3,\nn\\
&\displaystyle g_2=\frac{1}{12}-2c_3,~~~g_3=\frac{c_3}{6}-\frac{1}{216}-c_3^2\nn\eea
and with half-periods $\go_{1,2}$ depending on $c_3$.
\begin{figure}[h]
\begin{center}
\includegraphics[width=3.8cm]{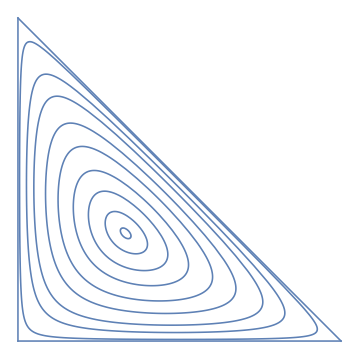}
\caption{The contours are controlled by $c_3$ while $s$ runs along each contour clockwise with period $2\pi$.}\label{fig_contour_intro}
\end{center}
\end{figure}
With the $c_3,s$ coordinates, the flow is simply a shift
\bea \phi_{\Gd t}:~s\to s+\frac{\pi\Gd t}{\go_1}\nn\eea
with $\Gd t$ the flow time. Our main result is that (the details are in props.\ref{prop_main}, \ref{prop_GKP} and \ref{prop_F})
\begin{theorem}
  On $\BB{C}P^2$, Hitchin's construction gives a Generalised K\"ahler structure with two toric invariant complex structures. The first one $I_-$ is the standard one while $I_+$ is deformed from $I_-$ by a flow $\phi_{\Gd t}$: $I_+=(\phi_{\Gd t*})^{-1}I_-\phi_{\Gd t*}$, concretely
  \bea
  && (I_+)^{\gt_i}_{~\,s}=(I_-)^{\gt_i}_{~s}\Big|_{s+\frac{\pi\Gd t}{\go_1}},~~(I_+)^{\gt_i}_{~\,c_3}=(I_-)^{\gt_i}_{~c_3}\Big|_{s+\frac{\pi\Gd t}{\go_1}}+\frac{3\tilde\eta_1\pi\Gd t}{\go_1^2c_3(1-27c_3)}(I_-)^{\gt_i}_{~s}\Big|_{s+\frac{\pi\Gd t}{\go_1}}\nn\\
  && (I_+)^s_{~\gt_i}=(I_-)^s_{~\gt_i}\Big|_{s+\frac{\pi\Gd t}{\go_1}}-\frac{3\tilde\eta_1\pi\Gd t}{\go_1^2c_3(1-27c_3)}(I_-)^{c_3}_{~\gt_i}\Big|_{s+\frac{\pi\Gd t}{\go_1}},~~(I_+)^{c_3}_{~\gt_i}=(I_-)^{c_3}_{~\gt_i}\Big|_{s+\frac{\pi\Gd t}{\go_1}}\nn\eea
  where the coordinate we use for $\BB{C}P^2$ are the two angles $\gt_{1,2}$ in addition to the $c_3,s$.
  The quantity $\tilde\eta_1$ is introduced in \eqref{eta_tilde} and depends only on $c_3$.

  The complex structure $I_+$ is valid for arbitrary flow time $\Gd t$, but the bi-hermitian metric is positive definite for small but non-zero $\Gd t$.
  The generalised K\"ahler potential is given by
  \bea
  & \displaystyle{K=\frac{\Gd t}{4}\log c_3+\frac{3\go_1}{8\pi}\int_0^{\Gd t\pi/\go_1}du\sum_{k=1}^3y^k(c_3,s+u)\log\frac{y^k(c_3,s+u)}{y^k(c_3,s)}}
.\nn\eea
The first term is proportional to the Fubini-Study K\"ahler potential.
\end{theorem}
We leave the further exploration of our explicit results, such as the relation to certain integrable models \cite{ruijsenaars1987} as well as mirror symmetry for a future work.

\bigskip
{\bf Acknowledgements:} We thank Maxim Zabzine for help over many issues in the draft. Our paper has some overlap with the work of Francis Bischoff and Marco Gualtieri, and we thank the former for sharing with us his thesis \cite{Bischoff} and both for many interesting discussions. The work of J.\,Q is supported in part by the grant  "Geometry and Physics"  from the Knut and Alice Wallenberg foundation.

\section{Generalised K\"ahler structure}\label{sec_GKS}
We will focus on 4-manifolds in this work. In this case the generalised K\"ahler structure (GKS) is equivalent to the bi-hermitian geometry $(M^4,g,I_{\pm})$, i.e. a Riemannian manifold with metric that is hermitian with respect to both (integrable) complex structures $I_{\pm}$.
Recall that for any (integrable) complex structure $J$ and a hermitian metric, there is a connection with torsion $\nabla^H$ with $\nabla^HJ=0$, where $\nabla^H$ is defined as
\bea \bra Z,\nabla^H_XY\ket=\bra Z,\nabla_XY\ket+ H(X,Y,Z).\nn\eea
The 3-form $H$ can be expressed in terms of $\go=J^*g$ as
\bea H=\frac12 d^c\go=\frac{i}{2}(\bar\partial\go-\partial\go).\nn\eea
We have now two complex structures and so two Hermitian 2-forms $\go_{\pm}=I_{\pm}^*g$, and two 3-forms $H_{\pm}$ constructed from $\go_{\pm}$ using the two complex structures $I_{\pm}$. It is proved in \cite{Biherm_cplx_surf} that if the first Betti number is even then $H_+=-H_-$, and so we will write $H$ for $H_+$ and $\nabla^{\pm}$ for the connection with torsion $\nabla^{\pm H}$. To summarise, a GKS consists of a bi-hermitian metric $(g,I_{\pm})$ with the covariant constancy condition
\bea \nabla^{\pm}I_{\pm}=0.\nn\eea

\subsection{Morita equivalence between holomorphic Poisson structures}
What played a crucial role in \cite{Bischoff:2018kzk} was the formulation of the GKS in terms of holomorphic Poisson structures.
One starts with the Hitchin's Poisson structure
\bea Q:=[I_+,I_-]g^{-1}.\label{Hitchin_Poisson}\eea
It can be shown that this is (0,2)+(2,0) with respect to both $I_{\pm}$. From $Q$ one can construct two Poisson structures
\bea \gs_{\pm}:=\frac{1}{4}(I_{\pm}Q+iQ),\label{two_poisson}\eea
which are (2,0) and holomorphic under $I_{\pm}$ respectively.
The authors of \cite{Bischoff:2018kzk} considered the GKS of symplectic type, we refer to \cite{Bischoff:2018kzk} for the details of this definition, while we only record what would be important for this work. One has a symplectic form $F$, with the following relations to $I_{\pm}$ 
\bea &I_+-I_-=-QF,\label{GKS_I}\\
&FI_++I^*_-F=0=I^*_-F+FI_+.\label{GKS_II}\eea
The solution to such a non-linear pair of equations is encoded in terms of the Morita equivalence of $\gs_{\pm}$ as holomorphic Poisson structures. We recall its definition following Ping Xu, adapted to the holomorphic setting
\begin{definition}\label{def_Ping_Xu}(Def 2.1 \cite{xu1991})
  The two holomorphic Poisson structures $(M,\gs_{\pm})$ are Morita equivalent if there
  is a holomorphic symplectic manifold $(X,J,\Go)$ (called the equivalence bi-module) and two maps ${\tt s}$, ${\tt t}$
\bea \begin{tikzpicture}
  \matrix (m) [matrix of math nodes, row sep=1.8em, column sep=1.5em]
    {   & (X,J,\Go) &  \\
       (M,I_+,\gs_+) & & (M,I_-,-\gs_-)  \\ };
  \path[->]
  (m-1-2) edge node[left] {\small${\tt t}$} (m-2-1)
  (m-1-2) edge node[right] {\small${\tt s}$} (m-2-3)  ;
  \end{tikzpicture}\nn\eea
 with the following properties. The source map ${\tt s}$ is holomorphic and anti-Poisson, i.e. ${\tt s}_*\Go^{-1}=-\gs_-$, while the target map ${\tt t}$ is holomorphic and Poisson i.e. ${\tt t}_*\Go^{-1}=\gs_+$. Besides, the fibre of $\tt{s},{\tt t}$ is connected and simply connected. One requires further that $\ker {\tt}_*$ be symplectic orthogonal complement to $\ker{\tt s}_*$.
\end{definition}
Here complete means that ${\tt s}^*f$ generates a complete Hamiltonian vector field (its flow time extendable to $(-\infty,\infty)$) if the function $f$ on $M$ generates under $\gs_-$ a complete vector field. The same goes for $\gs_+$ and ${\tt t}^*f$.
The symplectic orthogonality of $\ker {\tt s}_*$ and $\ker {\tt t}_*$ implies
\bea {\tt t}_*\Go^{-\#}{\tt s}^*={\tt s}_*\Go^{-\#}{\tt t}^*=0.\nn\eea
This ensures that the hamiltonian vector field generated by a function of type ${\tt t}^*h$ with $h\in C^{\infty}(M)$ has no effect on the image of the source map (and vice versa).
Finally that ${\tt s}$ and ${\tt t}$ are holomorphic means that the two complex structures $I_{\pm}$ are 'unified' as one single complex structure up on $X$.

The Morita equivalence is an equivalence relation between integrable Poisson manifolds. In particular, integrable Poisson manifolds are self Morita equivalent with self-equivalence bi-module given by the symplectic groupoid, which we will need next.

In the setting of GKS of symplectic type, this Morita equivalence between $(M,\gs_{\pm})$ can be constructed from deforming a self-Morita equivalence of $(M,\gs_-)$. Let $(X,J,\Go_0)$ be a holomorphic symplectic groupoid integrating the Poisson structure $\gs_-$, i.e. a holomorphic symplectic
manifold $(X,J,\Go_0)$ with maps
\bea \begin{tikzpicture}
  \matrix (m) [matrix of math nodes, row sep=2em, column sep=1.5em]
    {   & (X,J_0,\Go_0) &  \\
       (M,I_-,\gs_-) & & (M,I_-,-\gs_-)  \\ };
  \path[->]
  (m-1-2) edge node[left] {\small${\tt t}$} (m-2-1)
  (m-1-2) edge node[right] {\small${\tt s}$} (m-2-3)  ;
  \end{tikzpicture}\nn\eea
satisfying identical conditions as in def.\ref{def_Ping_Xu} by replacing $I_+,\,\gs_+$ with $I_-,\,\gs_-$.
One has again a similar orthogonality
\bea {\tt t}_*(\Go_0^{-1})^{\#}{\tt s}^*={\tt s}_*(\Go_0^{-1})^{\#}{\tt t}^*=0.\label{orthog_s_t}\eea
It is helpful to keep in mind that $J_0$ can be expressed as
\bea J_0=(\im\Go_0)^{-1}\re\Go_0\label{J_0_Omega}\eea
and that the inverse of $\Go_0$ (as a holomorphic 2-form) is
\bea \Go_0^{-1}=\frac14((\re\Go_0)^{-1}-i(\im\Go_0)^{-1}).\label{used_IX}\eea

It turns out that a simple deformation $\Go_0\to \Go=\Go_0+{\tt t}^*F$ gives us the Morita equivalence between $\gs_{\pm}$.
For simplicity of discussion we assume that ${\tt t}^*F$ is such that $\re\Go_0+t^*F$ remains invertible. First the kernel of $\Go$ defines an distribution $\bar L\subset T_{\BB{C}}X$ which is integrable since $\Go$ is closed. The invertibility of $\re\Go_0+{\tt t}^*F$ ensures that $L\cap \bar L=0$ and $L\oplus \bar L=T_{\BB{C}}X$. This means that $L$ is spanned by the (1,0) vector fields in $T_{\BB{C}}X$ but under a different complex structure $J\neq J_0$. The new complex structure can be expressed in the same way as \eqref{J_0_Omega}
\bea J=(\im\Go)^{-1}\re\Go=(\im\Go_0)^{-1}(\re\Go_0+{\tt t}^*F)=J_0+(\im\Go_0)^{-1}{\tt t}^*F.\label{J_Omega}\eea
The source map ${\tt s}$ remains holomorphic $(X,J)\to (M,I_-)$, indeed for any $V\in TX$
\bea {\tt s}_*JV={\tt s}_*J_0V+{\tt s}_*(\im\Go_0)^{-\#}{\tt t}^*F{\tt t}_*V=I_-{\tt s}_*V\nn\eea
where the second term drops thanks to \eqref{orthog_s_t}. In contrast, for the ${\tt t}$ map we have
\bea {\tt t}_*JV={\tt t}_*(J_0+(\im\Go_0)^{-\#}{\tt t}^*F)V=I_-{\tt t}_*V+{\tt t}_*(\im\Go_0)^{-\#}{\tt t}^*F{\tt t}_*V
=(I_--QF){\tt t}_*V,\nn\eea
where we have used the fact that ${\tt t}$ is a Poisson map sending $\Go_0^{-1}$ to $\gs_-=(I_-Q+iQ)/4$. The last equation says that ${\tt t}_*$ now intertwines $J$ with a certain automorphism $I_+$ of $TM$ with $I_+=I_--QF=I_-(1+I_-QF)$. This gives \eqref{GKS_I}.

Furthermore, that the lhs of \eqref{J_Omega} squares to $-1$ gives
\bea & J_0(\im\Go_0)^{-1}{\tt t}^*F+(\im\Go_0)^{-1}{\tt t}^*FJ_0+((\im\Go_0)^{-1}{\tt t}^*F)^2=0,\nn\\
& \To~~J_0^*{\tt t}^*F+{\tt t}^*FJ_0+{\tt t}^*F(\im\Go_0)^{-1}{\tt t}^*F=0,\nn\\
&\stackrel{\textrm{Apply }{\tt t}_*}{\To}~~I_-^*F+FI_--FQF=0\nn\eea
which is \eqref{GKS_II}. This also implies $(I_+)^2=-1$.

The Poisson property for ${\tt s}$ is unaffected, as for ${\tt t}$
\bea {\tt t}_*(\Go^{-1})=\frac14{\tt t}_*((\re\Go_0+{\tt t}^*F)^{-1}-i\im\Go_0^{-1})=\frac14({\tt t}_*(\re\Go_0+{\tt t}^*F)^{-1}+iQ).\nn\eea
To evaluate the first term we can use a geometric series expansion \cite{BursztynRadko} and get
\bea {\tt t}_*(\re\Go_0+{\tt t}^*F)^{-1}=(1+I_-QF)^{-1}I_-Q\nn\eea
provided $1+I_-QF$ is invertible. So we get
\bea {\tt t}_*(\Go^{-1})=\frac14((1+I_-QF)^{-1}I_-Q+iQ)=\frac14(I_+Q+iQ)=\gs_+\nn\eea
using \eqref{GKS_I}. This shows that ${\tt t}_*$ maps $\Go^{-1}$ to $\gs_+$.

\subsection{Encoding the generalised K\"ahler potential}
The Morita equivalence between $\gs_{\pm}$ encodes the generalised K\"ahler potential (GKP) as a kind of generating function (see  sec.4. of \cite{Bischoff:2018kzk}). Recall that just as the usual K\"ahler potential is not a globally defined function, the GKP is defined locally, and so the following discussion is local in nature.

From the previous section, one has a holomorphic symplectic form $\Go=\Go_0+{\tt t}^*F$, which can be expressed using local Darboux coordinate as $\Go=i\sum_kdP_k\wedge dQ^k$ where $Q_k,P^k$ are local holomorphic coordinates. As the Morita equivalence $(X,\Go,J)$ was constructed out of a symplectic groupoid $(X,\Go_0,J_0)$ integrating $\gs_-$, the space of units is a Lagrangian $L_0$ with respect to $\Go_0$. Though $L_0$ is no longer Lagrangian under $\Go$, one still has $\im\Go\big|_{L_0}=0$ since $F$ is real. Suppose that locally in the $P,Q$ coordinate system $L_0$ is given by $P_k=\eta_k(Q,\bar Q)$ where $\eta=\eta_k dQ^k$ is a 1-form, then one has
\bea d(i\eta_k dQ^k)=\Go|_{L_0}={\tt t}^*F.\nn\eea
This forces $\im i\eta$ to be a closed 1-form and hence $\re\eta=-d K$ locally for some real function $K(Q,\bar Q)$, as a consequence
\bea \eta=-2\partial K.\nn\eea
The local real function $K$ is the GKP.

As an example, for an actual K\"ahler manifold $(M,\go,I)$, one has $\gs_{\pm}=0$ and $I_+=I_-=I$. Then the Morita equivalence given simply by the holomorphic cotangent bundle $X=T^{(1,0)}M$ and $\Go_0$ is the standard holomorphic symplectic form. The maps ${\tt s}$, ${\tt t}$ are just the bundle projection $\pi$. The holomorphic symplectic form is then deformed to $\Go_0+\pi^*\go$, for which we can pick the Darboux coordinates as follows. Let $Q^k=x^k$ be the local holomorphic coordinates for $M$ and $\xi_k$ the local fibre coordinate, then $P_k=\xi_k-2\partial_kK$ with $K$ being the K\"ahler potential in the usual sense, i.e. $2i\partial\bar\partial K=\go$. We can easily check
\bea idP_k\wedge dQ^k=id(\xi_kdx^i-2\partial K)=id\xi_k\wedge dx^k-2i\bar\partial\partial K=\Go_0+\go.\nn\eea
Then the zero section $L_0=M$ is given by the condition $\xi_i=0$ or $P_i=-2\partial_iK$, which is how the GKP is formulated in the last paragraph.

\section{The Hitchin construction}
The pervading theme of the construction of the last section is that one starts from a K\"ahler manifold with $I_-$ as its complex structure, and one obtains $I_+$ as a deformation of $I_-$. In \cite{Hitchin_biherm} Hitchin used a flow to deform $I_-$ so that for small flow time, one gets a 4D bi-hermitian metric. Note that in this construction the flow time is small but cannot be zero.

\subsection{Overview}\label{sec_overview}
Let us give a quick overview of Hitchin's construction, while some proofs are collected in the next section.
One starts with a K\"ahler surface $(M,I_-,\gs_-)$, with a choice of a holomorphic Poisson tensor $\gs_-$ i.e. any element of $H^0(M,K^{-1})$ with $K$ the canonical bundle. The key observation is that the two holomorphic Poisson structures in \eqref{two_poisson} encoding the GKS share the same imaginary part $Q$, and so in deforming $I_-$ to reach $I_+$ one has to preserve $Q$.
It is thus natural to use a $Q$-Hamiltonian flow $\phi_t$ with $\phi_0=id$ generated by some Hamiltonian $h$ to deform $I_-$ so that
\bea I_+=(\phi_{t*})^{-1}I_-\phi_{t*}\label{I_t}\eea
is the new complex structure. The difficulty is to produce a positive definite bi-hermitian metric.

Hitchin picked the hamiltonian for the flow $h\sim\log ||\gs_-||^2$. Even though $h$ is ill-defined where $\gs_-$ vanishes, the hamiltonian vector field
\bea V=Q^{\#}dh\nn\eea
is globally defined, as demonstrated in \cite{Hitchin_biherm}. This is essentially because $Q$ is zero whenever $h$ diverges so that $V$, although not hamiltonian, is Poisson and its flow certainly preserves $Q$.

Take the inverse of $\gs_-$, considered as a meromorphic (2,0)-form, and denote with
$\go_-=\re(\gs_-^{-1})$. It is of course (2,0)+(0,2) with respect to $I_-$. Also its pull-back $\go_+=\phi_t^*\go_-$ is by construction $(2,0)+(0,2)$ with respect to $I_+$. But we consider its (1,1)-component with respect to $I_-$
\bea g=-I_-^*\go_++\go_+I_-.\label{Hitchin_metric}\eea
This 2-tensor turns out to be well defined (even though $\go_+$ is not), and symmetric (1,1) with respect to both $I_{\pm}$. But it may not be positive definite. To ensure its positivity, Hitchin studied the first derivative of $g$ with respect to $t$ and showed that
\bea I_-^*\dot g\big|_{t=0}=4i\partial \bar\partial h,\label{positivity}\eea
where $\partial$, $\bar\partial$ use the complex structure $I_-$.
But the rhs is the curvature of the line bundle $K^{-1}$, thus if the anti-canonical bundle is ample, then the curvature is positive definite. Since positivity is an open condition, for small but nonzero $t$ one has constructed a bi-hermitian metric.

\subsection{Collection of some formulae}
For clarity we temporarily denote with $I=I_-$ and $I_t=I_+$. The flow $\phi_t$ is generated by the vector field $V=Q^{\#}dh$, and so $I_t=I_+$ given in \eqref{I_t} can be written as
\bea I_t=I+\int_0^t dt \dot I_t=I+\int_0^tdt (\phi_{t*})^{-1}(L_VI)\phi_{t*},\nn\eea
while $L_VI$ can be written as $L_VI=-Qd(I^*dh)$ which we prove in lem.\ref{lem_used_II}.
Comparing with \eqref{GKS_I} one has
\bea -QF=I_t-I=-\int_0^tdt (\phi_{t*})^{-1}(Qd(I^*dh))\phi_{t*}=-Q\int_0^tdt (\phi_{t*})^{-1}d(I^*dh)\phi_{t*},\nn\eea
since $Q$ is preserved by $\phi_t$. So we get the formula for $F$
\bea F(t)=d\int_0^tdt \phi^*_t(I^*dh).\label{F}\eea
The first time derivative of $F(t)$ is
\bea \dot F(0)=dI^*dh=2i\partial\bar\partial h.\label{used_V}\eea

As for the metric, denote with $\gs_-=\gs$ and $\gs_+=\gs_t=(\phi_{-t})_*\gs$. Let $\go_+=\go_t=\re(\gs_t^{-1})$, and set as in \eqref{Hitchin_metric}
\bea g_t(X,Y)=-\go_t(IX,Y)+\go_t(X,IY).\nn\eea
At flow time $t=0$, $\go_0$ is (2,0)+(0,2) with respect to $I$ and so $g$ is zero. This is crucial for $g$ to be globally defined for $t>0$, since the divergence of $\go_t$ resides only in the (2,0)+(0,2) component.
We observe that
\bea g_t(X,Y)=-\go_t(IX,Y)+\go_t(X,IY)
=Q^{-1}(IX,I_tY)-Q^{-1}(I_tX,IY)\nn\eea
where we used $\go_t=\re(\gs_t^{-1})=-Q^{-1}I_t$ and that $Q$ is $(0,2)+(2,0)$ w.r.t. $I$.
We see from the last formula that $g_t$ is a symmetric tensor and
\bea g=Q^{-1}[I,I_t]\nn\eea
c.f. \eqref{Hitchin_Poisson}. We omit the subscript $t$ in $g$ next. We see a democracy between $I_t$ and $I$, thus if $g$ is (1,1) under $I$ so will it be under $I_t$.
Finally that $g$ is well-defined comes from the relation $I_t-I=-QF$
\bea g=Q^{-1}[I,-QF]=FI-I^*F.\nn\eea
Taking $t$-derivative at $t=0$ and using \eqref{used_V}, we get the positivity \eqref{positivity}.
\begin{lemma}\label{lem_used_II}
  \bea L_VI=-Qd(I^*dh).\label{used_II}\eea
\end{lemma}
\begin{proof}
  Let $U$ be any (0,1) vector field we compute $(L_VI)U=[V,IU]-I[V,U]=(-i-I)[V,U]$. To evaluate the rhs we pick a \emph{local} holomorphic function $f$ so that $df=\partial f$
  \bea \bra (L_VI)U,df\ket&=&-2i\bra [V,U],df\ket=-2i(V\circ U\circ f-U\circ V\circ f)=2iU\circ V\circ f\nn\\
  &=&2iU\circ Q(df,dh)=2iU\circ Q(\partial f,\partial h).\nn\eea
  As $Q^{(2,0)}$ and $f$ are holomorphic the $U$ derivative passes them by and gives
  \bea \bra (L_VI)U,df\ket=Q(\partial f,\iota_UdI^*dh)=-\bra df,(QdI^*dh)U\ket.\nn\eea
  The case for $U$ being $(1,0)$ is similar.
\end{proof}

\section{K\"ahler potential under the flow}
In this section we combine the two constructions of sec.\ref{sec_GKS} and sec.\ref{sec_overview}.
Let $(X,\Go_0)$ be holomorphic symplectic groupoid integrating $(M,I_-,\gs_-)$, then the flow $\phi_t$ generated by the vector field $V$ can be lifted up to $\hat\phi_t$ on $X$ generated by
\bea \hat V=-(\im\Go_0)^{-\#}({\tt t}^*dh),\nn\eea
where the minus sign comes from \eqref{used_IX} so that $\hat V$ covers $V$.
The lifted flow $\hat \phi_t$ satisfies
\bea {\tt t}\circ\hat\phi_t= \phi_t\circ {\tt t},~~~{\tt s}\circ\hat\phi_t={\tt s},\label{used_III}\eea
recall that the symplectic orthogonality \eqref{orthog_s_t} ensures that the image of ${\tt s}$ remains fixed, hence the second relation.

We have an expected result
\begin{lemma}
  \bea \Go_t=\hat\phi_t^*\Go_0=\Go_0+{\tt t}^*F.\label{used_IV}\eea
\end{lemma}
\begin{proof}
We let $\Go_t=\hat\phi_t^*\Go_0$ be the flow of $\Go_0$, then
\bea \dot\Go_t=\hat\phi_t^*L_{\hat V}\Go_0=\hat\phi_t^*L_{\hat V}((\im\Go_0)J_0))=(\im\Go_0)\hat\phi_t^*L_{\hat V}J_0\nn\eea
One has again the relation $L_{\hat V}J_0=(\im\Go_0)^{-1}dJ_0^*{\tt t}^*dh$ (here we understand the 2-form as a map from $T$ to $T^*$ and a bi-vector as a map from $T^*$ to $T$) similar to lem.\ref{lem_used_II}. Thus
\bea \dot\Go_t=\hat\phi_t^*dJ_0^*{\tt t}^*dh=\hat\phi_t^*{\tt t}^*dI^*dh={\tt t}^*\phi_t^*dI^*dh\nn\eea
where we used \eqref{used_III}. Continuing
\bea \Go_t-\Go_0=\int_0^t dt\dot\Go_t=\int_0^t {\tt t}^*\phi_t^*dI^*dh={\tt t}^* F\nn\eea
where $F$ was given in \eqref{F}.
\end{proof}

Next we look for the Darboux coordinates for $\Go_t$.
Pictorially the flow looks like fig.\ref{fig_cartoon}.
\begin{figure}[h]
\begin{center}
\begin{tikzpicture}[scale=1]
\draw[blue,line width=0.3mm] (-2,0) -- (2,0) node[right]{\scriptsize{$L_0\simeq M$}};
\draw [red,line width=0.3mm] plot [smooth, tension=0.5] coordinates {(-2,1) (-1.2,1.5) (1,0.8) (2,1.2)} node[right]{\scriptsize{$L_t$}};
\draw [->,densely dotted] plot [smooth, tension=0.5] coordinates {(-1.5,0) (-1.65,.4) (-1.55,.8) (-1.8,1.2)};
\node at (-1.4,0.6) {\scriptsize{$\hat\phi_t$}};
\node at (1.7,1.8) {\scriptsize{$X$}};

\draw[->] (1,.8) -- (1.2,0) node[below] {\scriptsize{${\tt s}$}};
\draw[->] (1,.8) -- (0.4,0) node[below] {\scriptsize{${\tt t}$}};

\draw[->] (0,1.1) -- (.6,1.1) node[right] {\scriptsize{$Q$}};
\draw[->] (0,1.1) -- (0,1.7) node[above] {\scriptsize{$P$}};

\draw[->] plot [smooth, tension=0.5] coordinates {(-0.4,0) (-0.5,0.3) (-0.6,0.7)};
\node at (-0.6,0.85) {\scriptsize{$Q_t$}};
\draw[->] plot [smooth, tension=0.5] coordinates {(-0.4,0) (-0.1,0.2) (0.3,0.4)};
\node at (0.45,0.4) {\scriptsize{$P_t$}};

\draw [->,densely dotted] plot [smooth, tension=0.5] coordinates {(-0.4,0) (-0.05,.5) (0,1.1)};
\node at (-0.15,0.5) {\scriptsize{$\hat\phi_t$}};
\end{tikzpicture}\caption{Cartoon of a flow $\hat \phi_t$ on $X$ lifting the flow $\phi_t$ on $M$.}\label{fig_cartoon}
\end{center}
\end{figure}
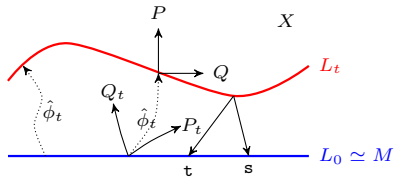
If we have found the Darboux coordinates for $\Go_0$
\bea i\sum_idP_i\wedge dQ^i=\Go_0\nn\eea
then these will pull back to
\bea Q_t=\hat\phi_t^*Q,~~P_t=\hat\phi_t^*P\nn\eea
and are subsequently the Darboux coordinates of $\Go_t$ since we have shown \eqref{used_IV}.
From this we also see that the new complex structure $J_t$ on $X$ induced from the deformed holomorphic symplectic structure $\Go_t$ is related to $J_0$ by the flow $\hat\phi_t$.

The space of units $L_0$ is flown to $L_t$. The Darboux coordinates $(P,Q)$ of $\Go_0$ pull back under $\hat\phi_t^*$.
Suppose that $L_t$ is described under $Q,P$ chart as
\bea L_t:=\{P_i=\eta_i(Q,\bar Q)\}\nn\eea
for some (1,0)-form $\eta$. Then as discussed earlier
\bea id\eta=\Go_0|_{L_t}=\hat\phi_t^*\Go_0|_{L_0}=(\Go_0+{\tt t}^*F)|_{L_0}={\tt t}^*F|_{L_0}\nn\eea
i.e. $\im(i\eta)$ is closed and one has a local real function $K(Q,\bar Q)$ with $i\eta|_{L_t}=-2i\partial K$, i.e. $L_t$ is the graph
\bea P_i=-2\partial_{Q^i}K.\nn\eea
There is in general no formula for finding the primitive for $\im(i\eta)$ apart from its $t$-derivative at $t=0$
\begin{proposition}\label{prop_used_III}
  Assuming that at $t=0$ the Lagrangian $L_0$ is described by $P=0$, i.e. $K\big|_{t=0}=0$, then
  the $t$-derivative of $K$ reads
  \bea \dot K\big|_{t=0}=-{\tt t}^*h.\nn\eea
\end{proposition}
The proof is immediate. In the toric case later, we will be able to get the primitive of $\re\eta$ somewhat more naturally.
%

%


\section{Application for toric Del Pezzo}
\subsection{Review of some elements of toric surfaces}
A good text book for this is Fulton \cite{Fulton:1436535}, and we shall also make use of the explicit parametrisation of toric manifolds due to Guillemin \cite{guillemin1994}.

Our focus of the toric geometry is more on its K\"ahler aspect, and so it is more appropriate to use the moment map polytope description (the fan description is more adapted for the algebraic aspect).
\begin{definition}
A polytope $\Delta\subset\BB{R}^d$ is the convex hull of a finite number of points in $\BB{R}^d$, called the vertices. A polytope $\Delta\subset\BB{R}^d$ is said to be \emph{Delzant} if for each vertex $p$ one has
\begin{enumerate}
 \item  there are exactly $d$ edges meeting at $p$;
 \item each edge meeting in $p$ is of the form $p+ t u$ with $u\in\BB{Z}^d$ and $t\geq 0$ ({\it rationality});
 \item the $d$ edges meeting in $p$ form a $\BB{Z}$-basis for $\BB{Z}^d$. (\emph{smoothness}) \label{smoothness}
\end{enumerate}
\end{definition}
\begin{theorem}\label{thm_Delzant}(Delzant \cite{Delzant_1988})
There is a 1-1 correspondence between symplectic toric manifolds (up to $\BB{T}^d$ equivariant symplectomorphisms)
and Delzant polytopes.
\end{theorem}
Given a toric symplectic manifold $M$, the image of the moment map for the torus action is a polytope and it is Delzant.
While from a Delzant polytope, one can construct a toric symplectic manifold $M_{\Gd}$ through symplectic reduction or symplectic cut \cite{Lerman1995}, the Delzant condition ensures the smoothness. It is useful to keep in mind the gist of the construction: one can visualise $M_{\Gd}$ as a $\BB{T}^d$ fibration over the interior of $\Gd$, however on a face $\Gd_a$ with normal $\vec v^a$ the torus action
\bea \vec v^a_i\frac{\partial}{\partial \phi_i}\nn\eea
degenerates. Here we use $\phi_i$ to denote the $d$ angle-coordinates of $\BB{T}^d$. Condition 3 above says that, each vertex has a neighbourhood of type $\BB{C}^d$, and up to an $SL(d,\BB{Z})$ the $\phi_i$ are the standard angle coordinates of $\BB{C}^d$.

For the remainder of the paper, we let $d=2$ and all polygons are assumed Delzant.
We denote the primitive normal to face $\Gd_a$ as $\vec v^a$. In fact, each face corresponds to an embedded $\BB{C}P^1$ in $M_{\Gd}$. It  is a toric invariant divisor, and all toric invariant divisors are generated (not freely) by the faces.
Suppose that the polygon is described by
\bea \Gd=\{\vec y\in \BB{R}^2\,|\,\bra \vec y,\vec v^a\ket+\gl^a\geq0\},\nn\eea
where $\vec y$ are the coordinates of $\BB{R}^2$.
Then from \cite{guillemin1994} the K\"ahler class is expressed as
\bea [\go]=\sum_a \gl^a[a].\nn\eea
Since the divisors $[a]$ are not independent and have $n$ relations among them, this expansion is not unique.
\begin{definition}
A smooth surface $M$ is called Del Pezzo if it is birational to $\BB{C}P^2$ and the anti-canonical class is ample.
The self intersection $d=K\cdotp K$ is called the degree of the Del Pezzo, here $K$ is the canonical divisor.
\end{definition}
Except for $\BB{C}P^1\times\BB{C}P^1$, the other Del Pezzo's are obtained from blowing up to 8 points in general position.
But we are interested in toric Del Pezzo's, so one can only blow up toric fixed points on $\BB{C}P^2$.
The polygon corresponding to $\BB{C}P^2$ is a triangle, see fig.\ref{fig_toric_del_pezzo}.
There are only three toric invariant points located at the three corners of the triangle.
Further blow-ups will product $-2$ curves and is not allowed (since any irreducible curve $C$ with $C\cdotp C<0$ must be rational and $C\cdotp C=-1$ if $-K$ is ample). We are left with the following polygons
\begin{center}
  \begin{tikzpicture}[scale=1]
\draw [-,blue] (0,0) -- (1,0) -- (0,1) -- (0,0);
\node at (0.5,-.3) {\scriptsize{$\BB{C}P^2$}};
\end{tikzpicture},~~
  \begin{tikzpicture}[scale=1]
\draw [-,blue] (0,0) -- (1,0) -- (1,1) -- (0,1) -- (0,0);
\node at (0.5,-.3) {\scriptsize{$\BB{C}P^1\times\BB{C}P^1$}};
\end{tikzpicture},~~
  \begin{tikzpicture}[scale=1]
\draw [-,blue] (0,0) -- (1.6,0) -- (0.6,1) -- (0,1) -- (0,0);
\node at (0.8,-.3) {\scriptsize{$\BB{F}_1$}};
\draw [-,red,thick] (0.6,1) -- (0,1);
\end{tikzpicture},~~
\begin{tikzpicture}[scale=1]
\draw [-,blue] (0,0) -- (1.2,0) -- (1.2,.5) -- (0.7,1) -- (0,1) -- (0,0);
\draw [-,red,thick] (1.2,0) -- (1.2,.5);
\draw [-,red,thick] (0.7,1) -- (0,1);
\node at (0.6,-.3) {\scriptsize{$d=7$}};
\end{tikzpicture},~~
\begin{tikzpicture}[scale=1]
\draw [-,red] (0,0.5) -- (0.5,0) -- (1.2,0) -- (1.2,0.5) -- (0.7,1) -- (0,1) -- (0,0.5);
\node at (0.6,-.3) {\scriptsize{$d=6$}};
\end{tikzpicture}\label{fig_toric_del_pezzo}
\end{center}
where we have marked the $-1$ curve in red.

\subsection{Action angle coordinates and complex coordinates}
Let $\vec y=(y^1,y^2)$ parametrise $\Gd\subset \BB{R}^2$, and let $\gt_1,\gt_2$ be the angle coordinates of $\BB{T}^2$. We have
\begin{proposition}\label{prop_symp_J}(Guillemin \cite{guillemin1994})
Let $M_{\Gd}$ be the toric symplectic manifold corresponding to $\Gd$, with normals $\vec v^a$.
The symplectic form of $M_{\Gd}$ reads
\bea \go=\sum_i dy^i\wedge d\gt_i.\nn\eea
Further the complex structure and metric read
\bea
& \displaystyle J=\sum_{ij}G_{ij}\partial_{\gt_i}\otimes dy^j - G^{ij}\partial_{y^i}\otimes d\gt_j,\label{comp_str}\\
& \displaystyle g=\sum_{ij} G_{ij}dy^i\otimes dy^j+G^{ij}d\gt_i\otimes d\gt_j\label{metric},
\eea
where
\bea\label{induced_metric}
&G=\displaystyle{\frac{1}{2}}\sum_a(\bra \vec v^a,\vec y\ket+\lambda^a)\log(\bra \vec v^a,\vec y\ket+\lambda^a),\nn\\
&\displaystyle G_i=\partial_{y^i}G,~~G_{ij}=\partial_{y^i}\partial_{y^j}G=\frac{1}{2}\sum_a\frac{\vec v^a_i\,\vec v^a_j}{\bra \vec v^a,\vec y\ket+\lambda^a}\nn\eea
and $G_{ij}\,G^{jk}=\delta_{ik}$.
The complex coordinates can be expressed as
\bea \xi_i=e^{i\gt_i+G_i}.\label{cplex_coord}\eea
\end{proposition}
Recall from the discussion below thm.\ref{thm_Delzant} that each corner of the polygon corresponds to a $\BB{C}^2$ patch, covering the whole $M$.
\begin{center}
\begin{tikzpicture}[scale=.8]
\draw [-,blue,thick]  (.8,2) -- (0,1) -- (0,0) -- (1.5,0) -- (3,1.5);
\draw [-,blue,dotted]  (0.8,2) -- (3,1.5);

\draw [->,blue] (0.3,1.4) node [left] {\scriptsize$\vec u^1$} -- (.7,1.05);
\draw [->,blue] (0,0.7) node [left] {\scriptsize$\vec u^2$} -- (.5,0.7);
\node at (0.25,1) {\scriptsize$z$};

\draw [->,blue] (1,0) node [below] {\scriptsize$\vec v^1$} -- (1,.5);
\draw [->,blue] (1.8,0.4) node [right] {\scriptsize$\vec v^2$} -- (1.5,0.7);
\node at (1.4,.2) {\scriptsize$\hat z$};

\end{tikzpicture}
\end{center}
At a corner with normals $\vec u^1,\vec u^2$, one can take combination of \eqref{cplex_coord} to give local complex coordinates for the $\BB{C}^2$ patch
\bea z_1=\xi_1^{u^2_2}\xi_2^{-u^2_1},~~~z_2=\xi_1^{-u^1_2}\xi_2^{u^1_1}.\label{local_cplx_coord}\eea
For later use we record how the complex coordinates change from one local patch to another.
\begin{lemma}\label{lem_coord_chg}
Suppose at the two corners of $\Gd$ with normals $\vec u^a$ and $\vec v^a$ respectively, one has complex coordinates $\{z_a,\,a=1,2\}$ and $\{\tilde z_a,\,a=1,2\}$ defined as in \eqref{local_cplx_coord} using $\{\vec u^a\}$ and $\{\vec v^a\}$ respectively, then over the open intersection where both coordinate systems are valid, one has the transformation
\bea \tilde z_a=\prod_b z_b^{A_{ab}},~~~A_{ab}=\left[
                                                 \begin{array}{cc}
                                                   \vec u^1\times\vec v^2 &\vec u^2\times\vec v^2 \\
                                                   \vec v^1\times \vec u^1 & \vec v^1\times\vec u^2 \\
                                                 \end{array}\right],~~~\det A=1.\label{chg_cplx_coord}\eea
\end{lemma}
The proof is a direct calculation. Of course in the derivation, we used $\vec u^1\times \vec u^2=1=\vec v^1\times \vec v^2$ since $\Gd$ is Delzant.

\subsection{Toric invariant holomorphic Poisson tensor}
A consequence of \eqref{chg_cplx_coord} is
\begin{lemma}
  The unique (up to a constant multiple) toric invariant holomorphic Poisson tensor on a toric surface has local expression
  \bea \gs=iz_1z_2\frac{\partial}{\partial z_1}\wedge \frac{\partial}{\partial z_2}.\label{Hol_Poisson_z}\eea
  This expression is defined on the $\BB{C}^2$ patch corresponding to a corner of $\Gd$ with normals $\vec u^1,\vec u^2$, but is in fact globally defined.
\end{lemma}
\begin{proof}
  We need only check that if one makes a coordinate transform as in \eqref{chg_cplx_coord}, the local expression of $\gs$ remains unchanged, but this follows from $\det A=1$.
\end{proof}
A straightforward calculation gives
\begin{lemma}
  When written in the action-angle coordinates, the holomorphic Poisson tensor $\gs$ reads
  \bea \gs=i(\frac12G^{1i}\partial_{y^i}-\frac{i}{2}\partial_{\gt_1})\wedge(\frac12G^{2j}\partial_{y^j}-\frac{i}{2}\partial_{\gt_2}).
  \label{Hol_Poisson_y}\eea
  Its norm computed with the metric \eqref{metric} is
  \bea ||\gs||^2=\frac12 {\det}^{-1}(G_{ij}).\label{norm_sigma}\eea
\end{lemma}
The determinant $\det(G_{ij})$ can be evaluated as
\bea 4\det(G_{ij})=(\sum_a \frac{v^a_1v^a_1}{\vec y\cdotp\vec v^a+\gl^a})
(\sum_b \frac{v^b_2v^b_2}{\vec y\cdotp\vec v^b+\gl^b})-(\sum_a \frac{v^a_1v^a_2}{\vec y\cdotp\vec v^a+\gl^a})^2.\nn\eea
The concrete expression for $\BB{C}P^2$ is given in sec.\ref{sec_GoCP2}.

\subsection{The symplectic groupoid}
According to the recipe for constructing the Morita equivalence, we need to find $(X,\Go_0)$ integrating $\gs$ given in \eqref{Hol_Poisson_z}. This has been studied in \cite{li2018symplectic} but, since this case is simple, we try to give a self contained presentation. We start from an example on $\BB{C}^2$.
\begin{example}\label{toy_model}
On $\BB{C}^2$ with complex coordinates $x,y$, we have a holomorphic Poisson tensor
\bea \pi=xy\partial_x\wedge \partial_y\label{second_model}.\eea
The pervading trick in finding the integrating object is to notice that $\pi$ is non-degenerate except when $xy=0$. This means that $(X,\Go)$ would look like a pair groupoid $X\sim U\times U$ in an open dense $U=\{xy\neq 0\}$, with
\bea \Go^{-1}=-\pi|_{(x,y)}+\pi|_{(\hat x,\hat y)}\nn\eea
where $x,y$ are the source coordinates and $\hat x,\hat y$ those of the target.

To extend to the whole of $\BB{C}^2$, we describe the symplectic groupoid as $\BB{C}^4$ with coordinates $x,y,u,v$ which away from the divisor, are related to the old ones as
\bea &\hat x/x=e^{uy},~~\hat y/y=e^{vx}.\label{change_of_coord}\eea
This trade is meant to reflect the fact that if $xy=0$ then $\hat x=x,\hat y=y$ necessarily and $u,v,x,y$ are a more fundamental set of coordinates.
Rewriting $\Go^{-1}$ in the $(x,y,u,v)$ coordinates
\bea -\Go^{-1}=-\partial_u\wedge\partial_y-uv\partial_u\wedge\partial_v-vy\partial_v\wedge\partial_y+\partial_v\wedge\partial_x
+ux\partial_u\wedge\partial_x+xy\partial_x\wedge\partial_y.\nn\eea
This is invertible, in fact $\det\Go=1$ everywhere. Its inverse is
\bea \Go=uvdx\wedge dy+vy dx\wedge du-dx\wedge dv+dy\wedge du-ux dy\wedge dv-xydu\wedge dv.\label{second_model_int}\eea
\end{example}

The next step is to globalise this construction by regarding $u,v$ as the fibre coordinates of some vector bundle to take into account the non-trivial coordinate change.
\begin{theorem}\label{thm_int_Poisson_toric}
  Let $M$ be a smooth toric surface associated with a Delzant polygon $\Gd$. The holomorphic Poisson structure $\gs$ \eqref{Hol_Poisson_z} can be integrated into a symplectic groupoid $(X,\Go_0)$. The space $X$ is diffeomorphic to the holomorphic cotangent bundle of $M$. If $(z^a,\xi_a)$ are the local base and fibre coordinates, then $\Go_0$ reads
  \bea i\Go_0=-\xi_1\xi_2 dz^1\wedge dz^2-\xi_1z^2 dz^1\wedge d\xi_2+dz^1\wedge d\xi_1+dz^2\wedge d\xi_2+\xi_2 z^1 dz^2\wedge d\xi_1-z^1z^2 d\xi_1\wedge d\xi_2\label{hol_symp_P2}.\eea
  Its inverse $\Pi=\Go_0^{-1}$ reads
  \bea -i\Pi=\partial_{\xi_1}\wedge\partial_{z^1}+\partial_{\xi_2}\wedge\partial_{z^2}-\xi_1\xi_2\partial_{\xi_1}\wedge\partial_{\xi_2}+\xi_1 z^2\partial_{\xi_1}\wedge\partial_{z^2}
 -\xi_2 z^1\partial_{\xi_2}\wedge\partial_{z^1}-z^1z^2\partial_{z^1}\wedge\partial_{z^2}\nn\eea
  Let $g=(z^1,z^2;\xi_1,\xi_2)$ be a point in $X$, the source, target maps read
  \bea& g=(z^1,z^2;\xi_1,\xi_2),~~{\tt s}(g)=(z^1,z^2),~~{\tt t}(g)=(z^1e^{\xi_2z^2},z^2e^{-\xi_1 z^1}).\nn\eea
  Let $h=(z^1e^{\xi_2z^2},z^2e^{-\xi_1z^1};\eta_1,\eta_2)$ be another point with ${\tt s}(h)={\tt t}(g)$ then
  \bea h\circ g=(z^1,z^2;\xi_1+\eta_1e^{\xi_2z^2},\xi_2+\eta_2e^{-\xi_1z^1}).\nn\eea
\end{theorem}
\begin{proof}
  We only need to prove that the local model ex.\ref{toy_model} globalises under the coordinate change rule lem.\ref{lem_coord_chg}.
  If $z^a\to\tilde z^a$ as in \eqref{chg_cplx_coord}, then as fibre coordinates of $T^*M$
  \bea \xi_a=\sum_b \frac{\tilde z^b}{z^a}A_{ba}\tilde \xi_b.\nn\eea
  This means the combination $z^a\xi_a$ (no sum) transforms linearly
  \bea z^a\xi_a=\sum_b \tilde z^b\tilde \xi_bA_{ba}.\nn\eea
  Note that $A_{ab}$ is an $SL(2,\BB{Z})$ matrix, but recall that for all $SL(2,\BB{C})$ matrices $A$ one has
  \bea \ep^{-1} A^T \ep =A^{-1},~~\ep=\left[
                                        \begin{array}{cc}
                                          0 & 1 \\
                                          -1 & 0 \\
                                        \end{array}\right]\nn\eea
  This shows that
  \bea \ep\left[
            \begin{array}{c}
              z^1\xi_1 \\
              z^2\xi_2 \\
            \end{array}\right]=\left[
            \begin{array}{c}
              z^2\xi_2 \\
              -z^1\xi_1 \\
            \end{array}\right]\nn\eea
  transform the same way as $\log z^a$, thus the target $(z^1e^{\xi_2z^2},z^2e^{-\xi_1 z^1})$ transform the same way as $(z^1,z^2)$.
  This shows that the ${\tt t}$ map is globally defined.

  Using the coordinate transform for $T^*\BB{C}^2$ it is straightforward, though tedious, to check that $\Go_0$ is globally well-defined.
\end{proof}

\begin{lemma}\label{lem_Darboux}
The Darboux coordinates are
\bea& q^1=z^1e^{\xi_2z^2/2},~~q^2=z^2e^{-\xi_1z^1/2},~~p_1=\xi_1 e^{-\xi_2 z^2/2},~~p_2=\xi_2 e^{\xi_1 z^1/2},\nn\\
&i\Go_0=dq^1 \wedge dp_1+dq^2 \wedge dp_2.\nn\eea
Under this coordinates, the source, target maps are
\bea && {\tt s}(q^1,q^2,p_1,p_2)=(q^1e^{-p_2q^2/2},q^2e^{p_1q^1/2}),~~~{\tt t}(q^1,q^2,p_1,p_2)=(q^1e^{p_2q^2/2},q^2e^{-p_1q^1/2}).\nn\eea
\end{lemma}
The proof is a direct computation.
\begin{remark}
From the Darboux coordinates it is easy to check
\bea \{{\tt t}^*f_1,{\tt s}^*f_2\}_{\Pi}=0,~~~f_{1,2}\in C^{\infty}(M)\nn\eea
which leads to \eqref{orthog_s_t}.

Lem.\ref{lem_Darboux} also gives a simpler check that the symplectic form \eqref{hol_symp_P2} is well-defined globally. Using the same argument as above, the $q^a,p_b$ transform as
  \bea \tilde q^a=\prod_b (q^b)^{A_{ab}},~~~\tilde p_a=(\tilde q^a)^{-1}\sum_b(p_bq^b)(A^{-1})_{ba}.\nn\eea
  Therefore using $\log q^a$ and $p_aq^a$ as coordinates, it is clear
  \bea \sum_ad\tilde q^a\wedge d\tilde p_a=\sum_ad\log \tilde q^a\wedge d(\tilde q^a\tilde p_a)
  =\sum_ad\log q^a\wedge d(q^ap_a)=\sum_adq^a\wedge dp_a.\nn\eea
\end{remark}

\section{GKS on $\BB{CP}^2$}\label{sec_GoCP2}
We put together the ingredients of previous sections and construct a GKS in explicit terms expressed with action-angle coordinates.
The polygon is a triangle with normals $\vec v^1=[1;0]$, $\vec v^2=[0;1]$ and $\vec v^3=[-1;-1]$.
We fix the K\"ahler class by fixing $\gl^1=0=\gl^2$ and $\gl^3=1$ so that $[\go]=[3]$.
\begin{center}
\begin{tikzpicture}[scale=1]
\draw [->] (-.2,0) -- (1.5,0) node[below] {\scriptsize{$y^1$}};
\draw [->] (0,-0.2) -- (0,1.5) node[right] {\scriptsize{$y^2$}};
\draw [-,blue] (0,0) -- (1,0) -- (0,1) -- (0,0);
\node at (1,-.2) {\scriptsize{$1$}};
\node at (-0.2,1) {\scriptsize{$1$}};
\draw [->] (0,0.5) node[left] {\scriptsize{$\vec v^1$}} -- (0.25,0.5);
\draw [->] (0.3,0) node[below] {\scriptsize{$\vec v^2$}} -- (0.3,0.25);
\draw [->] (0.5,0.5) node[right] {\scriptsize{$\vec v^3$}} -- (0.35,0.35);
\end{tikzpicture}
\end{center}

From the general formula prop.\ref{prop_symp_J}, we have
\bea
G=\frac12\sum_{i=1}^3 y^i\log y^i,~~~
G_{ij}=\frac12\left[
                \begin{array}{cc}
                  \frac{1-y^2}{y^1y^3} & \frac{1}{y^3} \\
                  \frac{1}{y^3} & \frac{1-y^1}{y^2y^3} \\
                \end{array}\right],~~~
G^{ij}=2\left[
                \begin{array}{cc}
                  y^1(1-y^1) & -y^1y^2 \\
                  -y^1y^2 & y^2(1-y^2) \\
                \end{array}\right].\nn\eea
where we have let $y^3=1-y^1-y^2$ which turns out quite convenient later on. Now
\bea 4\det(G_{ij})=\frac{1}{y^1y^2y^3}~~~\To~~~||\gs||^2=2y^1y^2y^3\nn\eea
from \eqref{norm_sigma}. This is completely expected since $\gs$ vanishes on a cubic curve, which in the toric invariant case is the product of three lines, corresponding to each of the three faces $\Gd_a=\{y^a=0\}$, $a=1,2,3$.

The complex coordinates \eqref{cplex_coord} are
\bea z^1=e^{i\gt_1+(\log y^1-\log y^3)/2}=(y^1/y^3)^{1/2}e^{i\gt_1},\nn\\
z^2=e^{i\gt_2+(\log y^2-\log y^3)/2}=(y^2/y^3)^{1/2}e^{i\gt_2},\label{cplex_coord_CP2}\eea
they are in fact the standard inhomogeneous coordinates for $\BB{C}P^2$.

The Hitchin Poisson tensor is the imaginary part of $\gs$
\bea Q=4\im \gs=4y^1y^2y^3\partial_{y^1}\wedge\partial_{y^2}-\partial_{\gt_1}\wedge\partial_{\gt_2}.\label{Hitch_Poisson_CP2}\eea

\subsection{Solving the flow}
We pick the Hamiltonian $h$ as
\bea h=-\frac{1}{4}\log (y^1y^2y^3).\nn\eea
The pre-factor is for convenience and can always be absorbed by redefining $t$.
This gives a vector field
\bea V=Q^{\#}dh=y^1(y^2-y^3)\frac{\partial}{\partial y^1}+y^2(y^3-y^1)\frac{\partial}{\partial y^2}\nn\eea
and so the flow equation reads
\bea \dot y^1(t)=y^1(t)(y^2(t)-y^3(t)),~~~\textrm{and cyc perm }(1\to 2\to 3)\label{flow_eqn}.\eea
It pays to keep the symmetry between $y^{1,2,3}$ and so we define the symmetric polynomials
\bea c_1=y^1+y^2+y^3=1,~~~c_2=y^1y^2+y^2y^3+y^3y^1,~~~c_3=y^1y^2y^3.\nn\eea
It is obvious that $c_3$ is a conserved quantity under the flow since $h\sim\log c_3$.
\begin{proposition}
  The flow equation \eqref{flow_eqn} is solved by
  \bea y^k(t)=\frac{c_3}{1/12-\wp(t+k\tf)}\label{sol_y}\eea
  with $\wp$ the Weierstrass elliptic function
  with $\go_{1,2}$ as the \emph{half} periods of a square lattice (fig.\ref{fig_lattice}) while $\tf=(2/3)\go_1$.

  Finally $t$ has constant imaginary part $-i\go_2$ and arbitrary real part \footnote{We made this awkward choice to avoid clutter in later formulae. A more pedantic way would be to write $t$ as $t+\go_2$ for real $t$. But as far as the flow is concerned one only cares about shifts of $t$, which is still real.}.
\end{proposition}
\begin{proof}
We can solve from $y^2+y^3=1-y^1$, $y^2y^3=c_3/y^1$ that
\bea y^2-y^3=\pm((1-y^1)^2-4c_3/y^1)^{1/2}\nn\eea
and hence
\bea (\dot y^1)^2=(y^1)^2((1-y^1)^2-4c_3/y^1)=y^1(y^1(1-y^1)^2-4c_3).\nn\eea
The rhs is a quartic polynomial with a root at $y^1=0$, which hints at changing variables $y^1=1/u$ so one can turn the quartic into a cubic
\bea (\dot u)^2=u^2-2u+1-4c_3u^3.\nn\eea
Further letting $\wp=1/12-c_3u$ gives
\bea \dot{\wp}^2=4\wp^3-g_2\wp-g_3=4\wp^3-(\frac{1}{12}-2c_3)\wp-(\frac{c_3}{6}-\frac{1}{216}-c_3^2).\nn\eea

The solution is the celebrated Weierstrass $\wp$ function. It is a doubly periodic meromorphic function with period lattice $2\Go$ generated by $2\go_1,2\go_2\in\BB{C}$. It is a convention that the lattice is $2\Go$ while $\go_{1,2}$ are the half periods, see fig.\ref{fig_lattice}.
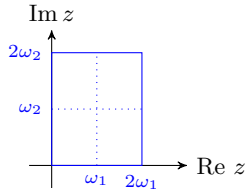
\begin{figure}[h]
\begin{center}
\begin{tikzpicture}
\draw [->] (-.3,0) -- (1.8,0) node [right] {\small$\re z$};
\draw [->] (0,-0.3) -- (0,1.8) node [above] {\small$\im z$};

\draw [-,blue] (0,0) -- (1.2,0) node [below] {\scriptsize$2\go_1$}  -- (1.2,1.5) -- (0,1.5) node [left] {\scriptsize$2\go_2$} -- (0,0);
\draw [dotted,blue] (0.6,0) node [below] {\scriptsize$\go_1$} -- (.6,1.5);
\draw [dotted,blue] (0,0.75) node [left] {\scriptsize$\go_2$} -- (1.2,0.75);

\end{tikzpicture}
\caption{A square lattice with periods $2\go_{1,2}$.}\label{fig_lattice}
\end{center}
\end{figure}
We compute the modular invariant \eqref{j_invariant} (see the appendix), which shows that the lattice is a square lattice.
For such a lattice, the real loci are the lattice lines of $\Go$.
We have the solution
\bea y^i(t)=\frac{c_3}{1/12-\wp(t+t_i)}\nn\eea
where $t_i$ is real and fixed using the initial data. In order that $y^i$ be real, the $t$ must have imaginary part equal to 0 or $\go_2$.
It cannot be the former option, since then $t+t_i$ can reach the point $\tf$ where $\wp(\tf)=1/12$ (see lem.\ref{lem_1/12}) and $y^i$ is unbounded. For the latter option, the solution is bounded and periodic with period $2\go_1$. Indeed starting from $0^+$ (where $\wp=+\infty$) going along straight lines via $\go_1$, $\go_1+\go_2$, $\go_2$ and reaching $i0^+$ (where $\wp=-\infty$), the function $\wp$ is monotonically decreasing. As $1/12$ is reached at $(2/3)\go_1$, one has $\wp<1/12$ along $z=\go_2$.
This is reviewed in sec.\ref{sec_TWwpf}.

To further relate $t_i$, we can analytically continue the solution down to the real axis and then approach $t+t_3=\tf$ from the right, where $y^3\to\infty$. But the conditions $\sum y^i=1$, $\prod y^i=c_3$ forces one of $y^{1,2}$ to go to $0^-$ and $-\infty$ respectively.
We pick the option $y^1\to-\infty$ and $y^2\to 0^-$ first. In order for this to happen one has to have
\bea t_1=t_3-2\tf\stackrel{\textrm{\scriptsize{mod~$2\Go$}}}{\sim} t_3+\tf,~~t_2=t_3-\tf.\nn\eea
Indeed, this way $t+t_1$ approaches $-\tf$ from the right giving $y^1\to-\infty$, and $t+t_2$ approaches $0$ from the right giving $y^2\to 0^-$. 

The other option would simply swap $t_1$ and $t_2$, to show that this is not the case, we look at fig.\ref{fig_contour_CP2}.
\begin{figure}[h]
\begin{center}
\includegraphics[width=3cm]{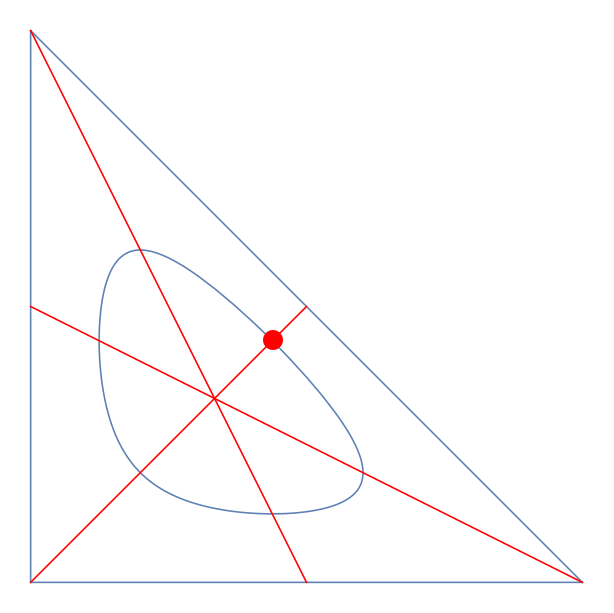}
\caption{The constant $c_3$ contour in the plane $y^{1,2}\in [0,1]$, $y^1+y^2\leq 1$.}\label{fig_contour_CP2}
\end{center}
\end{figure}
The figure shows the contour of fixed $c_3$, and the flow goes clockwise along this contour, as can be seen from \eqref{flow_eqn}.
The three red lines cut the triangle into six regions depending on the relative sizes of $y^i$.
With the first option, pick, say, $t_3=0$, i.e. one starts the flow at $y^1=y^2>y^3$ (the red dot). Increasing $t$, one has $y^3,y^1$ increase while $y^2$ decreases until one reaches $t=\go_2+\tf/2$, where $y^1>y^3=y^2$. This means one moves down into the east triangle and reaches the boundary with the southeast triangle at $t=\go_2+\tf/2$. However, had one picked the other option, one would swap $y^{1,2}$ and one would be moving up into the north east triangle. This does not agree with the direction of the flow.
\end{proof}
\begin{remark}
  It is instructive to see that the solution satisfies $\prod y^i=c_3$. Indeed, the product
  \bea (\frac1{12}-\wp(z))(\frac1{12}-\wp(z-\tf))(\frac1{12}-\wp(z+\tf))\nn\eea
  is doubly periodic without poles or zeros: for example, the first factor has a double pole at $z=0$, but the other two both have a single zero at $z=0$. Thus the product must be a constant, which can be obtained by setting $z$ at a convenient point. We choose $z=w_1$ and in lem.\ref{lem_special_point} we show that the constant is $c_3^2$.
\end{remark}
The solution involves the inversion of $\wp$ function, we would like to avoid that by expressing elliptic functions with Weierstrass zeta functions. These are recalled briefly in the appendix.
\begin{lemma}\label{lem_inv}
  One can express $y^i$ as Weierstrass zeta functions (where $\tf=2/3\, \go_1$)
  \bea && y^3(t)=\zeta(t-\tf)-\zeta(t+\tf)+2\zeta(\tf),\nn\\
  &&y^1(t)=\zeta(t)-\zeta(t-\tf)+\frac12-\zeta(\tf),\nn\\
  &&y^2(t)=\zeta(t+\tf)-\zeta(t)+\frac12-\zeta(\tf)\nn.\eea
\end{lemma}
\begin{proof}
We use the inversion formula (48.2 \cite{du_val_1973})
\bea \frac{\wp'(v)}{\wp(u)-\wp(v)}=\zeta(u-v)-\zeta(u+v)+2\zeta(v).\nn\eea
This is a special case of a broader formula that expresses any elliptic function using the Weierstrass $\zeta$ function, see thm.6.2. \cite{du_val_1973}. We show in lem.\ref{lem_1/12} that $\wp(\tf)=1/12$, and $\wp'(\tf)=-c_3$, the result for $y^3$ follows.
For $y^1$, one would get from the same formula that $y^1(t)=\zeta(t)-\zeta(t+2\tf)+2\zeta(\tf)$. Then one uses the periodicity $\zeta(z+2\go_i)=\zeta(z)+2\eta_i$ and also $\zeta(\tf)=1/6+2/3\,\eta_1$ from lem.\ref{lem_spec_value_zeta}.
\end{proof}

\subsection{The generalised K\"ahler potential}
In this section, we narrow down to $M=\BB{C}P^2$.
\begin{proposition}\label{prop_GKP}
  The GKP reads
  \bea K=\frac{t}{4}\log (c_3)+\frac{3}{8}\int_0^tdt\sum_iy^i_t\log\frac{y^i_t}{y^i_0}\label{GKP_CP2}\eea
  where $y^i_0=y^i(t_0)$, $y^i_t=y^i(t+t_0)$ and $y^i(z)$ is given in \eqref{sol_y}.
\end{proposition}
\begin{proof}
  Referring to fig.\ref{fig_cartoon}, the Darboux coordinates $Q,P$ for $\Go_0$ at $L_t$ will pullback under $\hat\phi_t$ to $Q_t,P_t$, the Darboux coordinates for $\Go_t$. Since
  \bea \Go_0=id\eta=id(P_idQ^i),\nn\eea
  we have $\re\eta$ is closed when restricted to $L_t$.

  To write down its real part in an efficient way, we exploit the fact that in an open dense subset $X$ is the pair groupoid. We can therefore use the source and target coordinates $z^a$ and $\hat z^a$ to parametrise $X$, and write
  $\Go_0$ and $\eta$ as follows. From the source and target maps in thm.\ref{thm_int_Poisson_toric}
  \bea \frac{\hat z^1}{z^1}=e^{P_2Q^2},~~z^1\hat z^1=(Q^1)^2,~~~\frac{\hat z^2}{z^2}=e^{-P_1Q^1},~~z^2\hat z^2=(Q^2)^2\nn\eea
  therefore
  \bea \eta&=&\sum_aP_adQ^a=\sum_a(P_aQ^a)d\log Q^a=-\frac12\log\frac{\hat z^2}{z^2}d\log (z^1\hat z^1)
  +\frac12\log\frac{\hat z^1}{z^1}d\log (z^2\hat z^2)\nn\\
  &=&\frac12(\log\hat z^{[1} d\log \hat z^{2]}-\log z^{[1} d\log z^{2]})+\frac12d(\log \hat z^1 \log z^2-\log z^1\log \hat z^2)\nn\eea
  where $[a\ b]$ means anti-symmetrisation of indices $a,b$.

  The second term is already exact, we focus on the first term and its real part
  \bea \re(\log\hat z^{[1} d\log \hat z^{2]}-\log z^{[1} d\log z^{2]})=
  \log |\hat z^{[1}|d\log |\hat z^{2]}|-\log|z^{[1}|d\log |z^{2]}|-\arg\hat z^{[1} d\arg \hat z^{2]}+\arg z^{[1} d\arg z^{2]}.  \nn\eea
  Note that in the $z,\bar z$ coordinates, restricting to $L_t$ means setting $\phi^*_t\hat z^a=z^a$. But we see from the flow equation \eqref{flow_eqn} and \eqref{cplex_coord_CP2} that the phases of the complex coordinates are constants under the flow. Therefore
  \bea 8\re\eta|_{L_t}=\log |\hat z^{[1}|^2 d\log |\hat z^{2]}|^2-\log |z^{[1}|^2 d\log |z^{2]}|^2+d(\log |\hat z^1|^2\log |z^2|^2-\log|z^1|^2\log|\hat z^2|^2).\label{used_I}\eea
  We focus on finding the primitive of the first term.
  We denote by
  \bea \xi=\log |z^{[1}|^2 d\log |z^{2]}|^2\nn\eea
  and the first term on the rhs of \eqref{used_I} is just $\phi_t^*\xi-\xi$. We have
  \bea \phi_t^*\xi-\xi=\int_0^t dt \dot\phi_t^*\xi=\int_0^t dt \phi_t^*L_V\xi=\int_0^t dt \phi_t^*(\iota_Vd+d\iota_V)\xi.\nn\eea
  We can evaluate explicitly
  \bea
  \iota_V\xi&=&\log|z^1|^2 (1-3y^1)+\log|z^2|^2 (1-3y^2)=\log\frac{y^1}{y^3}(1-3y^1)+\log\frac{y^2}{y^3} (1-3y^2)\nn\\
  &=&\sum_{i=1}^3(1-3y^a)\log y^a,\nn\\
  \iota_V d\xi&=&2\iota_V(d\log|z^1|^2\wedge d\log |z^2|^2)=2(3y^2-1)d\log\frac{y^2}{y^3}+2(3y^1-1)d\log\frac{y^1}{y^3}\nn\\
  &=&\frac{2(3y^2-1)}{y^2y^3}(y^2dy^1+(1-y^1)dy^2)+\frac{2(3y^1-1)}{y^1y^3}(y^1dy^2+(1-y^2)dy^1),\nn\\
  &=&2dy^1\frac{y^1-y^3}{y^1y^3}+2dy^2\frac{y^2-y^3}{y^2y^3}=-2d\log (y^1y^2y^3).\nn\eea
  In the computation we have used
  \bea |z^1|^2=\frac{y^1}{y^3},~~V\circ |z^1|^2=\frac{y^1}{y^3}(3y^2-1),\nn\\
  |z^2|^2=\frac{y^2}{y^3},~~V\circ |z^2|^2=\frac{y^2}{y^3}(1-3y^1).\nn\eea
  Altogether
  \bea \phi_t^*\xi-\xi=d\int_0^t(-3\sum_iy_t^i\log y_t^i-\log (y^1_0y^2_0y^3_0))\nn\eea
  recall that $y^1y^2y^3$ is a conserved quantity under the flow.

  The second term of \eqref{used_I} is already exact, but we can write it also as an integral. Let
  \bea g(t)=\log |\hat z^1|^2\log |z^2|^2-\log|z^1|^2\log|\hat z^2|^2\nn\eea
  differentiating and integrating
  \bea g(t)-g(0)=\int_0^t (3y^2_t-1)\log \frac{y^2_0}{y^3_0}+(3y^1_t-1)\log \frac{y^1_0}{y^3_0}
  =\int_0^t \sum_{i=1}^3 (3y^i_t-1)\log y_0^i.\nn\eea
  Putting everything together
  \bea 8\re\eta=d\int_0^tdt\big(-3\sum_iy_t^i\log\frac{y_t^i}{y^i_0}-2\log (y^1_0y^2_0y^3_0)\big).\nn\eea
\end{proof}
\begin{remark}
  First note the leading term of \eqref{GKP_CP2} agrees with prop.\ref{prop_used_III}.

  The GKP is not a global function, indeed our expression is ill-defined whenever $y^i=0$. However the ill-defined term resides only in the leading term which is proportional to the K\"ahler potential of the Fubini-Study metric. The correction term can be extended to $y^i=0$ because the log has always the combination $y^i_t/y^i_0$.

  Unfortunately we do not manage to perform the integral. A possible strategy is to relate $\gs$ to the theta functions as in \eqref{sigma_theta} and use the infinite product formula for the latter to deal with the logarithm.
\end{remark}
Just knowing the generalised K\"ahler potential is not enough, one needs to know the local complex coordinates $Q,\bar Q$.
But this can be read off from the source and target maps. Again let us stay away from the anti-canonical divisor where $\gs$ vanishes. Then
\bea Q^1=(\frac{y_0^1y_t^1}{y^3_0y^3_t})^{1/2}e^{i\gt_1},~~~Q^2=(\frac{y_0^2y_t^2}{y^3_0y^3_t})^{1/2}e^{i\gt_2}.\nn\eea
When $t=0$, they revert back to the standard complex coordinates.

\subsection{Explicit complex structures}
Our flow $\phi_t$ is solved by elliptic functions, so to compute e.g. the pullback map $\phi_t^*$, one needs to differentiate $\phi_t$ with respect to the initial points $y^{1,2}_0$, where the notation is as in prop \ref{prop_GKP}. These initial values determine $c_3$ and so affect the flow in two ways
\begin{enumerate}
  \item through the shift $\tf=2/3\go_1$ given in \eqref{used_VII}
  \item through the dependence of $\wp,\zeta,\gs$ on $c_3$ (via $g_2$ and $g_3$ as in \eqref{g23_ours}) given in \eqref{used_VI}.
\end{enumerate}
To compute these derivatives, it is beneficial to preserve the symmetry between $y^{1,2,3}$. So we opt to use the coordinates $c_3$ and the $t$ to parametrise the triangle base of $\BB{C}P^2$, see fig.\ref{fig_contour_CP2}, where $c_3$ determines the contours and $t$ parametrises each contour. However, the $t$-variable has a $c_3$-dependent period $2\go_1$, so we re-scale $t$ in order that all the contours are of period $2\pi$
\bea t=\go_2+\frac{s\go_1}{\pi}.\label{new_time_coord}\eea
This has a small price that $s$ is not a good variable for $c_3=0$ or $1/27$ where the $\go_1$ or $\go_2$ goes to infinity.
To summarise, the $c_3$ and $s$ coordinates are the 'polar' coordinates of the triangle, with $c_3\in[0,1/27]$ functioning as radius while $s\in[0,2\pi]$ as the angle.

The initial values $y^1_0,y^2_0$ of the flow will be expressed in terms of the new coordinates $c_3,s_0$.
We compute the following partial derivatives in the appendix (see lem.\ref{lem_partial_c_3})
\bea
 \partial_{c_3}y^i(s)=\frac{1}{c_3(1-27c_3)}\bigg\{\frac12((y^i)^2-y^i)-9c_3y^i+6c_3
+3y^i(y^{i+1}-y^{i+2})\vgs(s+(i-3)\frac{2\pi}{3})\bigg\},\nn\eea
where the index $i$ is taken mod 3 and the function $\vgs(z)$ is defined as
\bea \varsigma(z)=\zeta(\go_2+\frac{z\go_1}{\pi})-\eta_2-\frac{z}{\pi}\eta_1.\nn\eea
It is real of period $2\pi$ for \emph{real} $z$. See def.\ref{def_vgs} for further properties of this function.
With these we have
\begin{proposition}\label{prop_main}
  In the $c_3,s$ coordinate system, the Hitchin Poisson structure \eqref{Hitch_Poisson_CP2} reads
  \bea Q=\frac{4c_3\pi}{\go_1}\partial_{c_3}\wedge\partial_s-\partial_{\gt_1}\wedge\partial_{\gt_2}.\nn\eea
  The standard complex structure reads
  \bea &&\bra d\gt_i,J\partial_s\ket=\left[
       \begin{array}{cc}
         J^{\gt_1}_{~\,s} \\
         J^{\gt_2}_{~\,s} \\
       \end{array}\right]=\frac{\go_1}{2\pi}\left[
       \begin{array}{cc}
         3y^2-1 \\
         1-3y^1 \\
       \end{array}\right],~~~\bra dc_3,J\partial_{\gt_i}\ket=\left[
                                                               \begin{array}{cc}
                                                                 J^{c_3}_{~\gt_1} & J^{c_3}_{~\gt_2} \\
                                                               \end{array}
                                                             \right]=\frac{4\pi c_3}{\go_1}\left[
                                                 \begin{array}{cc}
                                                   -J^{\gt_2}_{~\;s} & J^{\gt_1}_{~\;s} \\
                                                 \end{array}\right],\nn\\
       &&\bra d\gt_i,J\partial_{c_3}\ket
                           =\left[
       \begin{array}{cc}
         J^{\gt_1}_{~\,c_3} \\
         J^{\gt_2}_{~\,c_3} \\
       \end{array}\right]=\frac{1}{4c_3(1-27c_3)}\left[
                           \begin{array}{c}
                             1+3y^2- 18 y^1y^2 -6\vgs(s)(1-3y^2)  \\
                             1+3y^1- 18 y^1y^2 +6\vgs(s)(1-3y^1)  \\
                           \end{array}\right],\nn\\
&&\bra ds,J\partial_{\gt_i}\ket=\left[\begin{array}{cc}
                                                 J^s_{~\gt_1} & J^s_{~\gt_2} \end{array}\right]=\frac{4c_3\pi}{\go_1}\left[
                                                 \begin{array}{cc}
                                                   J^{\gt_2}_{~c_3} & -J^{\gt_1}_{~c_3} \end{array}\right]\nn\eea
where $y^k(s)$ read
\bea y^k(s)=\frac{c_3}{\frac{1}{12}-\wp(\go_2+\frac{\go_1}{\pi}(s+\frac{2\pi k}{3}))}.\nn\eea
\end{proposition}
\begin{proof}
  The first statement is trivial, since $h=-1/4\log c_3$ is the hamiltonian that generates the flow $\partial_t$. Nonetheless, we can calculate it explicitly. The angular part of $Q$ is not affected, while the radial part is
  \bea 4c_3\partial_{y^1}\wedge\partial_{y^2}=4c_3\frac{\partial{(c_3,s)}}{\partial (y^1,y^2)}\partial_{c_3}\wedge \partial_s\nn\eea
  and the Jacobian can be computed using the explicit expressions $\partial_{c_3}y^i(s)$ given above. It turns out that the Jacobian equals
  \bea \frac{\partial{(c_3,s)}}{\partial (y^1,y^2)}=\frac{\pi}{\go_1}\nn\eea
  which reflects the fact that when $\go_1=0$ (which happens when $c_3=1/27$, the contour in fig.\ref{fig_contour_CP2} is the centre point) the $s$-coordinate is not valid.

  For the complex structure, since we have
  \bea J=G_{ij}\partial_{\gt_i}\otimes dy^j-G^{ij}\partial_{y^i}\otimes d\gt_j.\nn\eea
  We only need to focus on the first term, since the second is the inverse of the first.
  Evaluating the derivatives we get
  \bea \bra d\gt_i,J\partial_s\ket=\frac12\left[
       \begin{array}{cc}
         \frac{1-y^2}{y^1y^3} & \frac{1}{y^3} \\
         \frac{1}{y^3} & \frac{1-y^1}{y^2y^3} \\
       \end{array}\right]\left[
                           \begin{array}{c}
                             y^1(y^2-y^3) \\
                             y^2(y^3-y^1) \\
                           \end{array}\right]\frac{\go_1}{\pi}=\frac{\go_1}{2\pi}\left[
       \begin{array}{cc}
         3y^2-1 \\
         1-3y^1 \\
       \end{array}\right],\nn\eea
       while
       \bea \bra d\gt_i,J\partial_{c_3}\ket=\frac12\left[
       \begin{array}{cc}
         \frac{1-y^2}{y^1y^3} & \frac{1}{y^3} \\
         \frac{1}{y^3} & \frac{1-y^1}{y^2y^3} \\
       \end{array}\right]\left[
                           \begin{array}{c}
                             \partial_{c_3}y^1 \\
                             \partial_{c_3}y^2 \\
                           \end{array}\right]
                           =\frac{1}{4c_3(1-27c_3)}\left[
                           \begin{array}{c}
                             1+3y^2- 18 y^1y^2 -6\vgs(1-3y^2)  \\
                             1+3y^1- 18 y^1y^2 +6\vgs(1-3y^1)  \\
                           \end{array}\right].\nn\eea
The components $\bra ds,J\partial_{\gt_i}\ket$ and $\bra dc_3,J\partial_{\gt_i}\ket$ are fixed by demanding $J^2=-1$, giving
\bea
\bra dc_3,J\partial_{\gt_i}\ket=-2c_3\left[
                                                 \begin{array}{c}
                                                   1-3y^1 \\
                                                   1-3y^2 \\
                                                 \end{array}\right]^T,~~~
  \bra ds,J\partial_{\gt_i}\ket=-\frac{\pi}{(1-27c_3)\go_1}\left[
                                                 \begin{array}{c}
                                                   -1 - 3y^1 + 18y^1y^2 - 6\vgs(s)(1 - 3 y^1) \\
                                                   1+3y^2- 18y^1y^2 - 6\vgs(s)(1-3y^2) \\
                                                 \end{array}\right]^T.\nn\eea
\end{proof}
\begin{remark}\label{rmk_important}
With the explicit Poisson vector field $Q$ given above, the flow $\phi_t$ is simply
\bea \phi_t:\,(c_3,s)\to (c_3,s+\frac{t\pi}{\go_1}).\nn\eea
It is also crucial to remember that the flow time is measured in $t$, not in $s$. Even though $s$ and $t$ differ only by a factor of $\go_1$, the half period depends on $c_3$ in a quite complicated way \eqref{used_VII}. As a concrete instance where this can cause confusion we state the next corollary.
\end{remark}
\begin{corollary}\label{cor_two_I}
{\it  The two complex structures of the generalised K\"ahler structure of $\BB{C}P^2$ are explicitly given by $I_-(c_3,s)=J(c_3,s)$ and
  \bea
  && (I_+)^{\gt_i}_{~\,s}=J^{\gt_i}_{~s}(c_3,s+\frac{\pi\Gd t}{\go_1}),~~(I_+)^{\gt_i}_{~\,c_3}=J^{\gt_i}_{~c_3}(c_3,s+\frac{\pi\Gd t}{\go_1})+J^{\gt_i}_{~s}(c_3,s+\frac{\pi\Gd t}{\go_1})\frac{3\tilde\eta_1\pi\Gd t}{\go_1^2c_3(1-27c_3)}\nn\\
  && (I_+)^s_{~\gt_i}=J^s_{~\gt_i}(c_3,s+\frac{\pi\Gd t}{\go_1})-J^{c_3}_{~\gt_i}(c_3,s+\frac{\pi\Gd t}{\go_1})\frac{3\tilde\eta_1\pi\Gd t}{\go_1^2c_3(1-27c_3)},~~(I_+)^{c_3}_{~\gt_i}=J^{c_3}_{~\gt_i}(c_3,s+\frac{\pi\Gd t}{\go_1})\nn\eea
 where $\Gd t$ is the flow time and $J(c_3,s)$ is given in the proposition above.
  The terms proportional to $\Gd t$ above comes from the $c_3$-derivative of $s+\pi\Gd t/\go_1$.}
\end{corollary}

Next let us investigate the 2-form $F$ more closely.
\begin{proposition}\label{prop_F}
  For fixed flow time $\Gd t$, the 2-form $F$ in \eqref{GKS_I} reads
  \bea F=-(I^*_+d\gt_1-I^*_-d\gt_1)\wedge d\gt_2+[1\leftrightarrow2]\nn\eea
  with $I_{\pm}$ given in the last corollary.
\end{proposition}
This result checks \eqref{GKS_I} explicitly.
\begin{proof}
Fixing $\Gd t$ we have from \eqref{F}
\bea F=d\int_0^{\Gd t}dt \phi^*_t(I^*dh)=\frac{1}{2}d\int_0^{\Gd t}dt\phi_t^*((1-3y^1)d\gt_1+(1-3y^2)d\gt_2)=-\frac{3}{2}d\int_0^{\Gd t}dt\phi_t^*(y^i\wedge d\gt_i).\nn\eea
where the factor 3 is expected since $dd^ch$ is proportional to the curvature of the canonical class which is $3$ times the K\"ahler class. The integral
\bea Y^i=\int_0^{\Gd t}dt \phi_t^*y^i\nn\eea
can be worked out using lem.\ref{lem_inv} and the fact that the integral of $\zeta$ is $\log\gs$, see \eqref{sigma_zeta}.
\bea Y^1=\log\frac{\gs(\go_2+\frac{\go_1s}{\pi}+\Gd t)\gs(\go_2+\frac{\go_1s}{\pi}-\tf)}{\gs(\go_2+\frac{\go_1s}{\pi}+\Gd t-\tf)\gs(\go_2+\frac{\go_1s}{\pi})}+\Gd t(\frac{1}{2}-\zeta(\tf)),\nn\\
Y^2=\log\frac{\gs(\go_2+\frac{\go_1s}{\pi}+\Gd t+\tf)\gs(\go_2+\frac{\go_1s}{\pi})}{\gs(\go_2+\frac{\go_1s}{\pi}+\Gd t)\gs(\go_2+\frac{\go_1s}{\pi}+\tf)}+\Gd t(\frac{1}{2}-\zeta(\tf))\label{Y12}\eea
And so $F$ is simply
\bea F=-\frac{3}{2}\sum_{i=1}^2dY^i\wedge d\gt_i.\nn\eea
It is straightforward to compute $dY^i$ i.e. $\partial_{c_3}Y^i,\,\partial_sY^i$, though rather tedious. So we record only one intermediate step
\bea &&\partial_{c_3}\Big(\log\frac{\gs(\go_2+\frac{\go_1s}{\pi})}{\gs(\go_2+\frac{\go_1s}{\pi}-\tf)}+\frac{s\go_1}{\pi}(\frac{1}{2}-\zeta(\tf))\Big)\nn\\
&&~~~~=\frac{1}{c_3(1-27c_3)}\bigg\{\frac12(y^2-y^3)-\frac{s}{\pi}\tilde\eta_1+(3y^1-1)\vgs(s+2\pi/3)\bigg\}=-\frac{2}{3}J^{\gt_2}_{~\;c_3}(s)-\frac{s\tilde\eta_1}{\pi c_3(1-27c_3)},\nn\\
&&\partial_{c_3}\Big(\log\frac{\gs(\go_2+\frac{\go_1s}{\pi}+\tf)}{\gs(\go_2+\frac{\go_1s}{\pi})}+\frac{s\go_1}{\pi}(\frac{1}{2}-\zeta(\tf))\Big)\nn\\
&&~~~~=\frac{1}{c_3(1-27c_3)}\bigg\{\frac12(y^3-y^1)-\frac{s}{\pi}\tilde\eta_1+(3y^2-1)\vgs(s-2\pi/3)\bigg\}
=\frac{2}{3}J^{\gt_1}_{~\;c_3}(s)-\frac{s\tilde\eta_1}{\pi c_3(1-27c_3)},\nn\eea
where $\tilde \eta_i=\eta_i+\go_i(1/12-3c_3)$ and we have \emph{discarded} any $s$-\emph{independent} term.

The reason we can disregard such terms is that to get $\partial_{c_3}Y^i$ we need only take the difference of the above expression at $s+\pi\Gd t/\go_1$ and at $s$, so the $s$-independent terms drop. However we must not forget the $c_3$-dependence in the shift $\pi\Gd t/\go_1$, which produces $3y^i\Gd t\tilde\eta_1/(c_3(1-27c_3)\go_1)$. This is another instance where the remark \ref{rmk_important} is important.
\bea \partial_{c_3}Y^1&=&-\frac{2}{3}J^{\gt_2}_{~\;c_3}(s+\pi\Gd t/\go_1)+\frac{2}{3}J^{\gt_2}_{~\;c_3}(s)+(3y^1-1)\frac{\tilde\eta_1 \Gd t}{\go_1 c_3(1-27c_3)}\nn\\
&=&-\frac{2}{3}J^{\gt_2}_{~\;c_3}(s+\pi\Gd t/\go_1)+\frac{2}{3}J^{\gt_2}_{~\;c_3}(s)-J^{\gt_2}_{~\,s}\frac{2\pi\tilde\eta_1 \Gd t}{\go^2_1 c_3(1-27c_3)}
=-\frac23(I_+-I_-)^{\gt_2}_{~\;c_3}\nn\\
\partial_{c_3}Y^2&=&\frac{2}{3}J^{\gt_1}_{~\;c_3}(s+\pi\Gd t/\go_1)-\frac{2}{3}J^{\gt_1}_{~\;c_3}(s)+(3y^2-1)\frac{\tilde\eta_1 \Gd t}{\go_1 c_3(1-27c_3)}\nn\\
&=&\frac{2}{3}J^{\gt_1}_{~\;c_3}(s+\pi\Gd t/\go_1)-\frac{2}{3}J^{\gt_1}_{~\;c_3}(s)+J^{\gt_1}_{~\,s}\frac{2\pi\tilde\eta_1 \Gd t}{\go_1^2 c_3(1-27c_3)}=\frac23(I_+-I_-)^{\gt_1}_{~\;c_3},\nn\eea
where we have used cor.\ref{cor_two_I}.
The $s$-derivatives are much easier
\bea \partial_sY^1&=&\frac{\go_1}{\pi}(y^1(s+\pi\Gd t/\go_1)-y^1(s))=-\frac23(I_+-I_-)^{\gt_2}_{~\;s}\nn\\
\partial_sY^2&=&\frac{\go_1}{\pi}(y^2(s+\pi\Gd t/\go_1)-y^2(s))=\frac23(I_+-I_-)^{\gt_1}_{~\;s}\nn.\eea
Now the assembling is straightforward
\bea -F=(-I_++I_-)^{\gt_2}_{~\;s}ds\wedge d\gt_1-(I_+-I_-)^{\gt_2}_{~\;c_3}dc_3\wedge d\gt_1
+(I_+-I_-)^{\gt_1}_{~\;s}ds\wedge d\gt_2+(I_+-I_-)^{\gt_1}_{~\;c_3}dc_3\wedge d\gt_2\nn\eea
\end{proof}

\appendix
\section{Some facts of the elliptic functions}
Our primary source of reference is the book \cite{du_val_1973}.
\subsection{The Weierstrass $\wp$ function}\label{sec_TWwpf}
Given a lattice $\Go=\opn{span}_{\BB{Z}}\bra\go_1,\go_2\ket$, the function $\wp(z)$ is double periodic with period $2\Go$. It can be expressed as a sum
\bea \wp(u)=\wp(u|2\Go)=u^{-2}+\sum_{\go\in \Go}{'}(\frac{1}{(u-2\go)^2}-\frac{1}{(2\go)^2})\label{wp_sum}\eea
where $\Gs'$ means excluding the origin in the sum. The series is absolutely and uniformly convergent, once this is secured the double periodicity is immediate since one can shift the summation.

The function $\wp$ satisfies the well-known differential equation
\bea (\wp')^2=4\wp^3-g_2 \wp-g_3.\nn\eea
This function is even with a double pole with zero residue at $u=0+2\Go$. This fact plus Liouville's theorem show that the equation $\wp(u)=c$ has two roots (in one fundamental region). In particular, the three half periods $\go_1,\go_2,\go_3=\go_1+\go_3$ are the only points mod $2\Go$ where $\wp'=0$, if we denote
\bea &e_i=\wp(\go_i),\label{half_period_value}\\
&c_1(e_i)=e_1+e_2+e_3=0,~~c_2(e_i)=e_1e_2+e_2e_3+e_3e_1=-g_2/4,~~c_3(e_i)=e_1e_2e_3=g_3/4\nn\eea
then the equation $\wp(u)-e_i=0$ has a double root at $u=\go_i$.

In the following we assume
\bea g_2=\frac{1}{12}-2c_3,~~~g_3=\frac{c_3}{6}-\frac{1}{216}-c_3^2,~~c_3\in [0,1/27].\label{g23_ours}\eea
The discriminant of the cubic $4x^3-g_2x-g_3$ is
\bea \Gd=g_2^3-27g_3^2=c_3^3(1-27c_3).\nn\eea
One can compute the modular invariant
\bea j(\tau)=1728\frac{g_2^3}{\Gd}=\frac{(-1 + 24c_3)^3}{c_3^3(-1 + 27c_3)}\in[1728,\infty)\label{j_invariant}\eea
with the lowest value assumed at $c_3=(1/12 - \sqrt3/36)$.

If the $j$-invariant is between $(1,\infty)$, then the modular parameter $\tau$ corresponds to a square lattice (see fig.\ref{fig_fund_region}), which turns degenerate when $j(\tau)\to+\infty$.
For us this happens at $c_3=0$ or $c_3=1/27$, i.e. either $y^i=0$ (at the $i^{th}$ face of the polygon) or $y^{1,2,3}=1/3$ at the very centre.
\begin{figure}[h]
\begin{center}
\begin{tikzpicture}
  \draw[->] (-1.4,0) -- (1.4,0) node[right]{\scriptsize{$x$}};
  \draw[blue] (0.5,0.866) arc (60:120:1);
  \draw[black,dotted] (0.5,0.866) arc (60:0:1);
  \draw[black,dotted] (-0.5,0.866) arc (120:180:1);

  \draw[black,dotted] (.5,0.866) -- (.5,0) node[below]{\scriptsize{$\frac12$}};
  \draw[black,dotted] (-.5,0.866) -- (-.5,0) node[below]{\scriptsize{$-\frac12$}};

  \draw[-,blue] (0.5,0.866) -- (0.5,1.8);
  \draw[-,blue] (-0.5,0.866) -- (-0.5,1.8);

  \draw[-,red,line width=.4mm] (0,1) -- (0,2);
  \draw[->,black,dotted] (0,0) -- (0,2) node[right]{\scriptsize{$y$}};

  \node at (0.1,.8) {\scriptsize{$i$}};
\end{tikzpicture}
\caption{The fundamental region of $SL(2,\BB{Z})$ action on the upper plane. Our lattice is on the red line.}\label{fig_fund_region}
\end{center}
\end{figure}
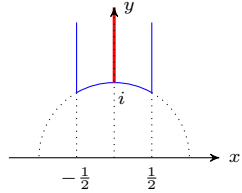

We assume from now on a square lattice
$\go_1\in\BB{R}_{>0}$ and $\go_2\in i\BB{R}_{>0}$. In such case the real locus of $\wp$ is along the lines $\re z=0,\go_1$ or $\im z=0,-i\go_2$ mod $2\Go$, which is easy to see from the summation \eqref{wp_sum}.
\begin{lemma}(sec.19 \cite{du_val_1973})
  Along the segment $[0,\go_1]$ along the real axis, the function $\wp$ goes from $+\infty$ to $e_1$ monotonically, while along the segment $[0,\go_2]$ along the imaginary axis, the function $\wp$ goes from $-\infty$ to $e_2$ monotonically. Further
  \bea e_1>e_3>e_2.\nn\eea
\end{lemma}
\begin{proof}
  As $\wp'=0$ only at $\go_1,\go_2,\go_3=\go_1+\go_2$ mod $2\Go$, the monotonicity will be clear once we figure out the relative sizes of $e_{1,2,3}$.

  From the pole $u^{-2}$ in \eqref{wp_sum}, one knows that going from $0$ to $\go_1$ along the $x$-axis $\wp$ decreases from $+\infty$ to $e_1$. Similarly going from 0 to $\go_2$ along the $y$-axis $\wp$ increases from $-\infty$ to $e_2$.

  Suppose $e_2>e_1$, then there is a value $c\in[e_1,e_2]$ where $\wp(u)=c$ has one root $u_1\in[0,\go_1]$ and another $u_2\in[0,\go_2]$. These are the only solutions.
  But $u_1+u_2$ should be $0\mod(2\Go)$ by thm 1.5. \cite{du_val_1973}, but this is impossible since $u_1$ is real and $u_2$ is imaginary. Thus $e_1>e_2$.

  By considering a path from $\go_1$ to $\go_3$ then to $\go_2$, if, say, $e_3>e_1>e_2$, then again there is $c\in[e_1,e_3]$ and (only) two roots to $\wp=c$ one in $[\go_1,\go_3]$, and another in $[\go_2,\go_3]$. This is again impossible using the same argument.
\end{proof}
We will need the value of $\wp$ at tertiary periods.
\begin{lemma}\label{lem_1/12}
  \bea \wp(\frac{2}{3}\go_1)=\frac1{12},~~~\wp'(\frac{2}{3}\go_1)=-c_3.\nn\eea
\end{lemma}
\begin{proof}
The four values
\bea \wp(\frac{2}{3}\go_i),~~i=1,2,3,4,~~\go_4=\go_1-\go_2\nn\eea
are the four roots of
\bea x^4-\frac12 g_2 x^2-g_3 x-\frac{1}{48}g_2^2=0.\label{quartic}\eea
For the proof see 15.1 \cite{du_val_1973}. This is true for arbitrary $g_2, g_3$, but if we plug in the concrete expression \eqref{g23_ours}, one gets
\bea x^4-\frac12 g_2x^2-g_3x-\frac{1}{48}g_2^2=(x-\frac{1}{12})\big((x+\frac{1}{36})^3+(c_3-\frac{1}{27})(x+\frac{1}{36})+c_3^2-\frac{c_3}{9}+\frac{2}{3^6}\big).\nn\eea
Thus one of $\wp(2/3\go_i)$ must be the root $1/12$. But $\wp(2/3\go_i)$ is not real for $i=3,4$, so we have either $\wp(2/3\go_1)=1/12$ with $1/12>e_1>e_3>e_2$, or $\wp(2/3\go_2)=1/12$ with $e_1>e_3>e_2>1/12$. But the last case would force $e_1+e_2+e_3>1/4$, but the sum should be zero. So we have $\wp(2/3\go_1)=1/12$. Note the case $e_1>1/12>e_2$ is impossible, since this would mean that $\wp=1/12$ must have a two roots $[\go_2,\go_2+2\go_1]$ or $[\go_1,\go_1+2\go_2]$, but $2/3\go_{1,2}$ are not in these segments.
\end{proof}
\begin{lemma}\label{lem_sixth}
Letting $\tf=2\go_1/3$, we have
\bea &&\wp(\tf/2)=\frac{1}{36c_3}(27c_3 + 6e_1 - 72c_3e_1 + 72e_1^2 - 1),\nn\\
&&\wp(\go_2+\tf)=\frac{1}{36c_3}(27c_3 + 6e_2 - 72c_3e_2 + 72e_2^2 - 1),\nn\\
&&\wp'(\tf/2)=-(3c_3+\wp(\tf/2)-\frac{1}{12}),\nn\\
&&\wp'(\go_2+\tf)=3c_3+\wp(\go_2+\tf)-\frac{1}{12}.\nn\eea
\end{lemma}
\begin{proof}
  Apply the addition formula to $\tf+(-\go_1)$
\bea \wp(\tf/2)=\wp(\tf-\go_1)=\frac14\frac{(\wp'(\tf))^2}{(\wp(\tf)-e_1)^2}-\wp(\tf)-e_1.\nn\eea
As everything on the rhs is known, we get
\bea
\wp(\tf/2)=\frac14\frac{c_3^2}{(1/12-e_1)^2}-1/12-e_1.\label{exp_sol}\eea
To simplify \eqref{exp_sol}, we note that $g_{2,3}$ are real, the $e_i$'s are algebraic over $\BB{R}$. By some elementary field theory, for any polynomial $f(e_i)$, one can express $1/f(e_i)$ as a polynomial of $e_i$. For example
\bea \frac{1}{1/12-e_i}=\frac{4}{c_3^2}(e_i^2+e_i/12+c_3/2-1/72).\nn\eea
Applying this to \eqref{exp_sol} we get
\bea \wp(\tf/2)=\frac{1}{36c_3}(27c_3 + 6e_1 - 72c_3e_1 + 72e_1^2 - 1).\nn\eea
As for the derivative, one can use directly \eqref{flow_eqn}, \eqref{sol_y}. Since $\wp(\tf/2)=1/12-c_3/y^1(-\tf/2)$ and so
$\wp'(\tf/2)=c_3(y^2-y^3)/y^1\big|_{t=-\tf/2}$. But at $-\tf/2$, one has $y^1=y^3$ and $y^2=1-2y^1$
\bea \wp'(\tf/2)=\frac{c_3}{y^1}(1-3y^1)=\frac{c_3}{y^1}-3c_3=\frac{1}{12}-\wp(\tf/2)-3c_3.\nn\eea

The calculation for $\wp(\go_2+\tf)$ unfolds by using the addition formula on $\go_2+\tf$, while for $\wp'(\go_2+\tf)$ we have
$\wp'(\go_2+\tf)=c_3(y^2-y^3)/y^1\big|_{t=\go_2}$. At $\go_2$, one has $y^1=y^2$ and $y^3=1-2y^1$
\bea \wp'(\go_2+\tf)=\frac{c_3}{y^1}(3y^1-1)=-\frac{c_3}{y^1}+3c_3=-\frac{1}{12}+\wp(\go_2+\tf)+3c_3.\nn\eea
\end{proof}

\begin{lemma}\label{lem_special_point}
  Assuming $\eqref{g23_ours}$, then
  \bea (1/12-\wp(\tf/2))^2(1/12-\wp(\go_1))=c_3^2.\nn\eea
\end{lemma}
\begin{proof}
This follows from the explicit values of $\wp(\tf/2)$ from lem.\ref{lem_sixth} and a direct calculation.
\end{proof}

\subsection{The Weierstrass $\zeta$ function}
The $\zeta$ function is defined as
\bea &\displaystyle{ \zeta(z)=\frac{1}{z}+\int_0^z(\frac{1}{z^2}-\wp(z))dz},\nn\\
&\displaystyle{ \wp'(z)=-\zeta(z)}.\label{wp_zeta}\eea
This definition is independent of the choice of the integral path, since $\wp$ has only double pole with zero residue at $z=0\mod2\Go$.

Unlike $\wp$, the function $\zeta$ is an odd function, and it is \emph{not} periodic.
One has instead (see thm.6.1 \cite{du_val_1973})
\bea &\zeta(z+2\go_1)=\zeta(z)+2\eta_1,~~\zeta(z+2\go_2)=\zeta(z)+2\eta_2,\nn\\
&\eta_1\go_2-\eta_2\go_1=i\pi/2.\label{pairing}\eea

\begin{lemma}\label{lem_spec_value_zeta}
  We have the value of $\zeta$ at some special points
  \bea & \displaystyle{\zeta(\tf)=\frac16+\frac23\eta_1,~~\zeta(2\tf)=-\frac16+\frac43\eta_1}\nn\\
   & \displaystyle{\zeta(\go_2+\tf)=-\frac{1}{c_3}(2e_2^2+e_2/6+c_3-1/36)+\eta_2+\frac16+\frac23\eta_1},\nn\eea
   where $\eta_{1,2}=\zeta(\go_{1,2})$.
\end{lemma}
\begin{proof}
To compute $\zeta(\tf)$, we use the duplication formula
\bea \zeta(2\tf)-2\zeta(\tf)=\frac12\frac{\wp''(\tf)}{\wp'(\tf)}=-\frac12\nn\eea
where we computed $\wp''$ by differentiating $(\wp')^2=4\wp^3-g_2\wp-g_3$:
\bea 2\wp'\wp''=12\wp^2\wp'-g_2\wp'~\To~\wp''=6\wp^2-\frac{1}{2}g_2.\nn\eea
But since $2\tf=2\go_1-\tf$, one gets $\zeta(2\tf)=2\eta_1+\zeta(-\tf)=2\eta_1-\zeta(\tf)$ as $\zeta$ is odd. Thus
$(2\eta_1-\zeta(\tf))-2\zeta(\tf)=-1/2$.

For the last statement, use again the addition formula
  \bea \zeta(\go_2+\tf)=\frac{1}{2}\frac{c_3}{e_2-1/12}+\eta_2+\zeta(\tf)=-\frac{2}{c_3}(e_2^2+e_2/12+c_3/2-1/72)+\eta_2+\frac16+\frac23\eta_1.\nn\eea
\end{proof}
\begin{definition}(54.1 \cite{du_val_1973})
  Define a new function
  \bea \zeta^*(z)=\zeta(z)-\frac{\eta_1}{\go_1}z.\nn\eea
  This is odd and periodic along the real axis.
\end{definition}
\begin{definition}\label{def_vgs}
  The function
  \bea \varsigma(z)=\zeta(\go_2+\frac{z\go_1}{\pi})-\eta_2-\frac{z}{\pi}\eta_1.\nn\eea
  is real of period $2\pi$ for real $z$. It satisfies the shift relations
  \bea \vgs(z+\frac{2\pi}{3})=\vgs(z)+y^2(t)-\frac13,~~~\vgs(z-\frac{2\pi}{3})=\vgs(z)-y^1(t)+\frac13,~~~t:=\go_2+\frac{z\go_1}{\pi}\nn.\eea
\end{definition}
\begin{proof}
  Using the addition formula
  \bea \zeta(\frac{z\go_1}{\pi})=\frac12\frac{\wp'(\go_2+\frac{z\go_1}{\pi})+\wp'(\go_2)}
  {\wp(\go_2+\frac{z\go_1}{\pi})-\wp(\go_2)}+\zeta(\go_2+\frac{z\go_1}{\pi})-\zeta(\go_2)\nn\eea
  where the lhs and the first term on the rhs are real for real $z$, besides $\eta_1$ is also real, so the reality of $\vgs$ follows.
  The periodicity follows from that of $\zeta^*$ defined earlier.

  For the shift we only show one of them,
  \bea &&\zeta(t+\tf)=\frac{1}{2}\frac{\wp'(t)+c_3}{\wp(t)-1/12}+\zeta(t)+\zeta(\tf)=-\frac{1}{2}(y^1-y^2+y^3)+\zeta(t)+\zeta(\tf)\nn\\
  \To&&\vgs(z+\frac{2\pi}{3})-\vgs(z)=y^2(t)-\frac13.\nn\eea
\end{proof}

\subsection{The Weierstrass $\gs$ function and theta function}
Since $\zeta$ has a simple pole at $0\mod 2\Go$, if one integrates $\zeta$, one gets a branching behaviour similar to that of a log.
Indeed $\int\zeta$ is the logarithm of an analytic function
\bea \gs(z)=z\exp\int_0^z(\zeta(z)-\frac{1}{z})dz.\label{sigma_zeta}\eea
It follows trivially that
\bea \frac{\gs'(z)}{\gs(z)}=\zeta(z).\label{sigma_zeta_I}\eea
Clearly $\gs$ is not periodic
\bea \frac{\gs(z+p\go_1+q\go_2)}{\gs(z)}=(-1)^{pq+p+q}e^{2(p\eta_1+q\eta_2)(z+p\go_1+q\go_2)}.\nn\eea
In fact, from this non-periodicity, one can relate $\gs$ function to the more familiar theta functions according to 18.10.8 \cite{abramowitz2012handbook}
\bea &\displaystyle { \gs(z)=\frac{2\go_1}{\pi}\exp(\frac{\eta_1z^2}{2\go_1})\frac{\vgt_1(\pi z/2\go_1)}{\vgt'_1(0)} },\label{sigma_theta}\\
&\displaystyle { \vgt_1(z)=\vgt_1(z;\tau)=2q^{1/4}\sum_{n\geq0}(-1)^nq^{n(n+1)}\sin(2n+1)z,~~q=e^{i\pi\tau},~~\tau=\go_2/\go_1},~~\Gd>0.\nn\eea
The $\vgt$ function is the solution to the heat equation with $\tau$ serving as the time, and so one has
the important property
\bea -i\pi\frac{\partial^2\vgt_1(z;\tau)}{\partial z^2}=4\frac{\partial\vgt_1(z;\tau)}{\partial\tau}\nn\eea
which helps us to compute the derivation of $\gs$ function with respect to the modular parameter $\tau$, and also $g_2,g_3$.

\subsection{Derivation with respect to $g_{2,3}$}
From the relation between $\gs$ and $\vgt_1$ plus the heat equation satisfied by $\vgt_1$, one can derive
\begin{lemma}(18.6.23-18.6.24 \cite{abramowitz2012handbook})
The derivative of $\gs$ with respect to $g_{2,3}$ reads
\bea
&&\frac{\partial\log\gs}{\partial g_2}=\frac{1}{16(g_2^3-27g_3^2)}(4g_2^2(z\zeta-1)+36g_3(\wp-\zeta^2)-3g_2g_3z^2),\nn\\
&&\frac{\partial\log\gs}{\partial g_3}=\frac{1}{8(g_2^3-27g_3^2)}(z^2g_2^2+12\zeta^2g_2-12\wp g_2-36g_3(z\zeta-1)).\nn\eea
The same expression for $\wp,\zeta$ can be obtained by taking derivatives and using \eqref{wp_zeta} and \eqref{sigma_zeta}.
\end{lemma}
In our situation $g_{2,3}$ depend only on $c_3$, and $\partial_{c_3}=(-2\partial_{g_2}+(1/6-2c_3)\partial_{g_3})$. We record for the readers convenience
\bea \partial_{c_3}\log\gs(z)&=&
-\frac{c_3^2}{96(g_2^3-27g_3^2)}(144\wp + (24c_3-1)z^2 - 144\zeta^2 + (864c_3-24)(\zeta z-1)),\nn\\
\partial_{c_3}\wp(z)
&=&\frac{c_3^2}{12(g_2^3-27g_3^2)}\big(6\wp(1-36c_3) + 3\wp'(12\zeta +z - 36c_3z)  + 72\wp^2 + 24c_3  - 1\big),\label{used_VI}\\
\partial_{c_3}\zeta(z)&=&-\frac{c_3^2}{48(g_2^3-27g_3^2)}(72\wp'+(\zeta-\wp z)(432c_3-12) + 144\wp\zeta + (24c_3-1)z ).\nn\eea
where $g_2^3-27g_3^2=c_3^3(1-27c_3)$.
\begin{lemma}
 We have
  \bea \partial_{c_3}\left[
                       \begin{array}{c}
                         \eta_i \\
                         \go_i \\
                       \end{array}\right]=\frac{1}{c_3(1-27c_3)}\left[
                       \begin{array}{cc}
                         \frac14(1-36c_3) & \frac{1}{48}(1-24c_3) \\
                         -3 & -\frac14(1-36c_3)\\
                       \end{array}\right]\left[
                       \begin{array}{c}
                         \eta_i \\
                         \go_i \\
                       \end{array}\right],~~i=1,2\label{used_VII}\eea
\end{lemma}
Note the matrix above is traceless and so it will preserve the condition \eqref{pairing}.
\begin{proof}
  For $\go_1$, since $\tf=2/3\go_1$, we can differentiate $\wp(\tf)=1/12$ and get $\partial_{c_3}\wp|_{t=\tf}+\wp'(\tf)\partial_{c_3}\tf=0$, then we get from \eqref{used_VI}
  \bea \partial_{c_3}\tf=\frac{1}{c_3}\partial_{c_3}\wp|_{t=\tf}
  =\frac{c_3}{12(g_2^3-27g_3^2)}\big(\frac12(1-36c_3) - 3c_3(12\zeta(\tf) +\tf - 36c_3\tf)  + \frac12 + 24c_3  - 1\big)\nn\eea
  we get the result.

  For $\eta_1$, one can differentiate $\zeta(\tf)=1/6+2/3\eta_1$, while for $\go_2$ and $\eta$ one differentiates $\wp(\go_2)=e_2$ and the second equation of lem.\ref{lem_spec_value_zeta}.
\end{proof}
One can define the combination
\bea \tilde \eta_i=\eta_i+\go_i(\frac{1}{12}-3c_3),\label{eta_tilde}\eea
then \eqref{used_VII} reads
\bea \partial_{c_3}\tilde\eta_i=-2\go_i,~~\partial_{c_3}\go_i=\frac{-3\tilde\eta_i}{c_3(1-27c_3)}.\label{ode_periods}\eea
In particular $\go_1$ satisfies $\partial_{c_3}(c_3(1-27c_3)\partial_{c_3}\go_1)=6\go_1$
and the solution is given by the hyper-geometric functions \cite{clemensscrapbook}.
We list the values of $\tilde \eta_{1,2}$ and $\go_{1,2}$ at $c_3=0,1/27$ for completeness
\bea c_3=0,~~\go_1=\infty,~~\go_2=i\pi,~~\tilde\eta_1=\frac12,~~\eta_2=-\frac{i\pi}{12},\label{lim_value_1}\\
c_3=\frac{1}{27},~~\go_1=\sqrt3\pi,~~\go_2=i\infty,~~\eta_1=\frac{\sqrt{3}\pi}{36},~~\tilde\eta_2=-\frac{i\sqrt3}{6}.\label{lim_value_2}\eea
If we write $c_3=\ep$ then $\go_1\sim \ep^{-1/2}$ in the first line, or if we write $1/27-c_3=\ep$ then $\go_2\sim i\ep^{-1/4}$ for the second line.

Finally we compute the $c_3$ derivative of the functions $y^{1,2,3}$. We have demonstrated in the main text that, in order to exploit the symmetry between $y^{1,2,3}$, we should use the $c_3$ and $t$ coordinate to parametrise the contour in fig.\ref{fig_contour_CP2}. In particular we introduce the time coordinate $s$
\bea t=\go_2+\frac{s\go_1}{\pi},\label{used_VIII}\eea
so that $s\in\BB{R}$ and is of period $2\pi$ regardless of $c_3$.
\begin{lemma}\label{lem_partial_c_3}
  Holding $s$ in \eqref{used_VIII} fixed, and differentiate with respect to $c_3$
  \bea
  \partial_{c_3}y^i=\frac{1}{c_3(1-27c_3)}\bigg\{\frac12((y^i)^2-y^i)-9c_3y^i+6c_3
+3y^i(y^{i+1}-y^{i+2})\vgs(s+(i-3)\frac{2\pi}{3})\bigg\},\nn\eea
where the index $i$ is taken mod 3 and $\vgs(s)$ is given in def.\ref{def_vgs}.
\end{lemma}
By using the shift relation in def.\ref{def_vgs}, one can check that $\partial_{c_3}(y^1+y^2+y^3)=0$ and $\partial_{c_3}(y^1y^2y^3)=1$.
\begin{proof}
  We hold $s$ fixed when taking the partial derivative wrt $c_3$.
  \bea \partial_{c_3}(y^3(t))=\partial_{c_3}\frac{c_3}{1/12-\wp(t)}+\dot y^3(t)\partial_{c_3}t
  =\frac{y^3(t)}{c_3}+\frac{c_3}{(1/12-\wp)^2}\partial_{c_3}\wp\big|_{t=\go_2+s\go_1/\pi}+y^3(y^1-y^2)\partial_{c_3}t\nn\eea
  One plugs in \eqref{used_VI} for $\partial_{c_3}\wp$ and \eqref{used_VII} for $\partial_{c_3}\go_i$,
  \bea \partial_{c_3}(y^3(t))&=&\frac{y^3}{c_3}+y^3(y^1-y^2)\partial_{c_3}t\nn\\
  &&+\frac{(1/12-\wp)^{-2}}{12(1-27c_3)}\bigg\{6\wp(t)(1-36c_3)+3y^1y^2(y^1-y^2)(12\zeta(t)+t(1-36c_3))+72\wp(t)^2+24c_3-1\bigg\}\nn\\
&=&\frac{y^3}{c_3}+\frac{y^3(y^1-y^2)}{4c_3(1-27c_3)}\bigg\{-12\eta_2-\go_2(1-36c_3)
  -\frac{s}{\pi}(12\eta_1+\go_1(1-36c_3))\bigg\}\nn\\
&&+\frac{(1/12-\wp)^{-2}}{12(1-27c_3)}
\bigg\{6(\wp-1/12)(3-36c_3)+3y^1y^2(y^1-y^2)(12\zeta(t)+t(1-36c_3))\nn\\
&&+72(\wp(t)-1/12)^2+6c_3\bigg\}\nn\\
&=&\frac{y^3}{c_3}+\frac{y^3(y^1-y^2)}{4c_3(1-27c_3)}\bigg\{-12\eta_2-\cancel{t(1-36c_3)}-\frac{12s}{\pi}\eta_1\bigg\}\nn\\
&&+\frac{1}{c_3(1-27c_3)}\bigg\{-\frac12y^3(3-36c_3)+\frac14y^3(y^1-y^2)(12\zeta(t)+\cancel{t(1-36c_3)})+6c_3+\frac12(y^3)^2\bigg\}\nn\\
&=&\frac{1}{c_3(1-27c_3)}\bigg\{\frac12((y^3)^2-y^3)-9c_3y^3+6c_3
+3y^3(y^1-y^2)(\zeta(\go_2+s\go_1/\pi)-\eta_2-s\eta_1/\pi)\bigg\}.\nn\eea
We replace the last term with the function $\vgs$ given in def.\ref{def_vgs}. The computation for $y^{1,2}$ are similar since they are obtained from $y^3$ by shifting $s\to s\pm 2/3\pi$.
\end{proof}
The Jacobian can be computed as
\bea \frac{\partial{(c_3,s)}}{\partial (y^1,y^2)}=\frac{\pi}{\go_1}.\nn\eea
Finally we perform an amusing calculation of the area of the triangle $\Gd=\{y^{1,2}\in[0,1],\;y^1+y^2\leq 1\}$
\bea \int_{\Gd} dy^1\wedge dy^2
=\int_{\Gd}\frac{\go_1}{\pi}dc_3\wedge ds=2\int_0^{1/27}\go_1dc_3
\stackrel{\eqref{ode_periods}}{=}-\tilde\eta_1\big|_0^{1/27}=\frac12.\nn\eea
This is admittedly a pretentious way of doing some grade school level maths, though a reassuring check of the extensive computations involving elliptic functions.

\bibliographystyle{alpha}

\begin{thebibliography}{JGMHR84}

\bibitem[AGG98]{Biherm_cplx_surf}
Vestislav Apostolov, Paul Gauduchon, and Gueo Grantcharov.
\newblock Bihermitian structures on complex surfaces.
\newblock {\em Proceedings of The London Mathematical Society - PROC LONDON
  MATH SOC}, 79, 04 1998.

\bibitem[AS12]{abramowitz2012handbook}
M.~Abramowitz and I.A. Stegun.
\newblock {\em Handbook of Mathematical Functions: with Formulas, Graphs, and
  Mathematical Tables}.
\newblock Dover Books on Mathematics. Dover Publications, 2012.

\bibitem[BGZ18]{Bischoff:2018kzk}
Francis Bischoff, Marco Gualtieri, and Maxim Zabzine.
\newblock {Morita equivalence and the generalized K\"ahler potential}.
\newblock 2018.

\bibitem[Bis19]{Bischoff}
Francis Bischoff.
\newblock Morita equivalence and generalized kahler geometry.
\newblock {\em University of Toronto thesis}, 2019.

\bibitem[BR03]{BursztynRadko}
Henrique Bursztyn and Olga Radko.
\newblock Gauge equivalence of dirac structures and symplectic groupoids.
\newblock {\em Annales de l'Institut Fourier}, 53(1):309--337, 2003.

\bibitem[Cle]{clemensscrapbook}
C.H. Clemens.
\newblock {\em A Scrapbook of Complex Curve Theory}.
\newblock Department of Mathematics, University of Utah.

\bibitem[Del88]{Delzant_1988}
Thomas Delzant.
\newblock Hamiltoniens p\'eriodiques et images convexes de l'application
  moment.
\newblock {\em Bulletin de la Soci\'et\'e Math\'ematique de France},
  116(3):315--339, 1988.

\bibitem[DV73]{du_val_1973}
Patrick Du~Val.
\newblock {\em Elliptic Functions and Elliptic Curves}.
\newblock London Mathematical Society Lecture Note Series. Cambridge University
  Press, 1973.

\bibitem[Ful93]{Fulton:1436535}
William Fulton.
\newblock {\em {Introduction to toric varieties}}.
\newblock Annals of mathematics studies. Princeton Univ. Press, Princeton, NJ,
  1993.

\bibitem[Gua14]{Gualtieri2014}
Marco Gualtieri.
\newblock Generalized K{\"a}hler geometry.
\newblock {\em Communications in Mathematical Physics}, 331(1):297--331, Oct
  2014.

\bibitem[Gui94]{guillemin1994}
Victor Guillemin.
\newblock Kaehler structures on toric varieties.
\newblock {\em J. Differential Geom.}, 40(2):285--309, 1994.

\bibitem[Hit03]{GCY}
Nigel Hitchin.
\newblock {Generalized Calabi–Yau Manifolds}.
\newblock {\em The Quarterly Journal of Mathematics}, 54(3):281--308, 09 2003.

\bibitem[Hit07]{Hitchin_biherm}
Nigel Hitchin.
\newblock Bihermitian metrics on del pezzo surfaces.
\newblock {\em J. Symplectic Geom.}, 5(1):1--8, 03 2007.

\bibitem[Hul86]{HULL1986357}
C.M. Hull.
\newblock Compactifications of the heterotic superstring.
\newblock {\em Physics Letters B}, 178(4):357 -- 364, 1986.

\bibitem[JGMHR84]{GHR}
S~J.~Gates, C~M.~Hull, and Martin Roček.
\newblock Twisted multiplets and new supersymmetric nonlinear sigma models.
\newblock {\em Nucl. Phys. B}, 248, 01 1984.

\bibitem[Ler95]{Lerman1995}
Eugene Lerman.
\newblock {Symplectic Cuts}.
\newblock {\em Mathematical Research Letters}, 2(3):247--258, 1995.

\bibitem[LR18]{li2018symplectic}
Songhao Li and Dylan Rupel.
\newblock Symplectic groupoids for cluster manifolds.
\newblock 2018.

\bibitem[LRvUZ07]{Lindstrom:2005zr}
Ulf Lindstrom, Martin Rocek, Rikard von Unge, and Maxim Zabzine.
\newblock {Generalized Kahler manifolds and off-shell supersymmetry}.
\newblock {\em Commun. Math. Phys.}, 269:833--849, 2007.

\bibitem[Rui87]{ruijsenaars1987}
S.~N.~M. Ruijsenaars.
\newblock Complete integrability of relativistic calogero-moser systems and
  elliptic function identities.
\newblock {\em Comm. Math. Phys.}, 110(2):191--213, 1987.

\bibitem[Str86]{STROMINGER1986253}
Andrew Strominger.
\newblock Superstrings with torsion.
\newblock {\em Nuclear Physics B}, 274(2):253 -- 284, 1986.

\bibitem[SYZ96]{Strominger:1996it}
Andrew Strominger, Shing-Tung Yau, and Eric Zaslow.
\newblock {Mirror symmetry is T duality}.
\newblock {\em Nucl. Phys.}, B479:243--259, 1996.

\bibitem[Xu91]{xu1991}
Ping Xu.
\newblock Morita equivalence of poisson manifolds.
\newblock {\em Comm. Math. Phys.}, 142(3):493--509, 1991.

\bibitem[Zum79]{ZUMINO1979203}
B.~Zumino.
\newblock Supersymmetry and K\"ahler manifolds.
\newblock {\em Physics Letters B}, 87(3):203 -- 206, 1979.

\end{thebibliography}

\end{document}